\documentclass[12pt]{article}
 
\usepackage{amstext,amssymb,amsmath,stmaryrd,amsthm}
\usepackage{geometry} 
\usepackage{color}
\usepackage{hyperref}
\usepackage{makeidx}
\makeindex

\numberwithin{equation}{section}

\newtheorem{theorem}{Theorem}[section]

\newtheorem{lemma}[theorem]{Lemma}
\newtheorem{proposition}[theorem]{Proposition}

\newtheorem{rem}[theorem]{Remark}

\begin{document}

\title{\Large{The Logarithmic Sobolev Inequality in Infinite dimensions  for Unbounded Spin Systems on the Lattice with non Quadratic Interactions.}}

\author{Ioannis Papageorgiou \thanks{   Department of Mathematics, Imperial College London, 180 Queen's Gate, London, SW7 2AZ. Email: ioannis.papageorgiou05@imperial.ac.uk  }
  }

\date{}
\maketitle
\begin{abstract}We are interested in the Logarithmic Sobolev Inequality for the infinite volume Gibbs measure with no quadratic interactions. We consider unbounded spin systems on the one dimensional Lattice with interactions that go beyond the usual strict convexity and without uniform bound on the second derivative. We assume that the one dimensional single-site measure with boundaries satisfies the Log-Sobolev inequality uniformly on the boundary conditions and we  determine conditions under which  the Log-Sobolev Inequality can be extended to the infinite volume Gibbs measure.
\end{abstract}


~



%

\section{Introduction}
We are interested in the $q$ Logarithmic Sobolev Inequality  (LSq) for measures related to systems of unbounded spins on the one dimensional Lattice with nearest neighbour interactions that are not strictly convex. Suppose that the Log-Sobolev Inequality is true for the single site measure
with a constant uniformly bound on the boundary conditions. The aim of this paper
is to present a criterion under which the inequality can be extended to the  infinite  volume Gibbs measure. More specifically, we extend the already
 know results for interactions $V$ that satisfy $\left\Vert \nabla_i \nabla_j V(x_i,x_j) \right\Vert_{\infty}<\infty$ to the more general case of interactions
with
$\left\Vert \nabla_i \nabla_j V(x_i,x_j) \right\Vert_{\infty}=\infty$.

 Regarding the Log-Sobolev Inequality for the local specification
 $\{\mathbb{E}^{\Lambda,\omega}\}_{\Lambda\subset\subset \mathbb{Z}^d,\omega \in
\Omega}$ on a d-dimensional Lattice, criterions and examples of measures   $\mathbb{E}^{\Lambda,\omega}$ that satisfy
the Log-Sobolev -with a constant uniformly on the set $\Lambda$ and the boundary
conditions $\omega-$ are investigated in [Z2], [B-E], [B-L], [Y] and [B-H].  For  
 
  $\left\Vert \nabla_i \nabla_j V(x_i,x_j) \right\Vert_{\infty}<\infty$ the Log-Sobolev is proved  when the phase
$\phi$ is strictly convex and convex at infinity. Furthermore, in [G-R]
the Spectral Gap Inequality is proved to be true for phases beyond the convexity at infinity, while in [M-M] and [B-J-S] the Decay of Correlation is studied.

For
the measure  $\mathbb{E}^{\{i\},\omega}$ on the real line, necessary
and sufficient  conditions are presented in  [B-G], [B-Z] and [R-Z], so that
  the Log-Sobolev Inequality is satisfied uniformly on the boundary
conditions $\omega$.

The problem
of the Log-Sobolev inequality for the Infinite dimensional Gibbs measure on the
Lattice is examined
in [G-Z], [Z1] and [Z2]. The first two study the LS for measures
on a d-dimensional Lattice for bounded spin systems, while the third one
looks at continuous spins systems on the one dimensional Lattice.

In [M] and [O-R],  criterions are presented  in order to pass from the Log-Sobolev
Inequality for the single-site measure $\mathbb{E}^{\{i\},\omega}$
to the LS2 for the Gibbs measure $\nu_{N}$ on a finite N-dimensional product space. Furthermore, using these criterions  one can conclude  the Log-Sobolev Inequality for the family $\{\nu_N,N\in\mathbb{N}\}$ with a constant uniformly
on $N$. Concerning the same problem for the LSq ($q\in (1,2]$) inequality in the case of Heisenberg groups with
quadratic interactions in 
[I-P]  a similar criterion is presented for the Gibbs  measure based on the methods
developed in [Z1] and [Z2].

All the pre mentioned developments refer to measures with interactions $V$
that satisfy 
$\left\Vert \nabla_i \nabla_j V(x_i,x_j) \right\Vert_{\infty}<\infty$. The question
that arises is whether similar assertions can be  verified for the infinite
dimensional Gibbs
measure in the case where  $\left\Vert \nabla_i \nabla_j V(x_i,x_j) \right\Vert_{\infty}= \infty$ and in this paper we present a strategy to solve this problem.

Consider the one dimensional measure
$$\mathbb{E}^{\{i\},\omega}(dx_{i})=\frac{e^{-
 \phi(x_i)-\sum_{j\sim
i}J_{ij}V(x_i
,\omega_j)}dX_i} {Z^{\{i\},\omega}} \text{\; with \;} \left\Vert\partial_x \partial_y V(x,y) \right\Vert_{\infty}= \infty$$          
 Assume that $\mathbb{E}^{\{i\},\omega}$ satisfies the (LS) inequality with a constant
uniformly on $\omega
$.
 Our aim is to set conditions, so that the infinite volume
Gibbs measure  $\nu$ for the local specification $\{\mathbb{E}^{\Lambda,\omega}\}_{\Lambda\subset\subset \mathbb{Z},\omega \in
\Omega}$    satisfies the LS inequality. We will focus on  measures
on the the one dimensional Lattice,  but our result can also be easily extended
on trees.   

 Our general setting is as follows:
 
 \textit{The Lattice.} When we  refer to the  Lattice we mean
 the 1-dimensional 
Lattice $\mathbb{Z}$. 
 
\textit{The Configuration space.}
We consider continuous unbounded random variables in $\mathbb{R}$, representing spins. Our configuration space is $\Omega=\mathbb{R}^{\mathbb{Z}}$. For any
$\omega\in \Omega $ and $\Lambda \subset \mathbb{Z}$ we denote
$$\omega=(\omega_i)_{i\in \mathbb{Z}}, \omega_{\Lambda}=(\omega_i)_{i\in\Lambda},\omega_{\Lambda^c}=(\omega_i)_{i\in\Lambda^c} \text{\; and \;}\omega=\omega_{\Lambda}\circ\omega_{\Lambda^c}$$
where $\omega_i \in \mathbb{R}$. When $\Lambda=\{i\}$ we will write $\omega_i=\omega_{\{i\}}$.
Furthermore, we will write $i\sim j$ when the nodes $i$ and $j$     are nearest neighbours, that means, they are connected with a vertex,
while
 we will denote the set of the neighbours of $k$ as   $\{\sim k\}=\{r:r \sim k\}$.
 
\textit{The functions of the configuration.}
We consider integrable functions $f$ that depend on a finite set of variables
$\{x_i\}, i\in{\Sigma_f}$ for a finite subset $\Sigma_f\subset\subset \mathbb{Z}$. The symbol $\subset\subset$ is used to denote a finite subset.

\textit{The Measure on $\mathbb{Z}$.}
For any subset $\Lambda\subset\subset
\mathbb{Z}$ we define the probability measure
$$\mathbb{E}^{\Lambda,\omega}(dx_\Lambda) =
\frac{e^{-H^{\Lambda,\omega}}dx_\Lambda} {Z^{\Lambda,\omega}}$$ 
  where
 \begin{itemize}
\item $x_{\Lambda}=(x_i)_{i\in\Lambda}$ and $dx_\Lambda=\prod_{i\in\Lambda}dx_i$

\item $Z^{\Lambda,\omega}=\int e^{-H^{\Lambda,\omega}}dx_\Lambda$
\item
$H^{\Lambda,\omega}=\sum_{i\in\Lambda}
\phi (x_i)+\sum_{i\in\Lambda,j\sim
i}J_{ij}V(x_i
,z_j)$
 
\end{itemize}
\noindent
and
\begin{itemize}
\item
$z_j=x_{\Lambda}\circ\omega_{\Lambda^c}=\begin{cases}x_j & ,i\in\Lambda  \\
\omega_j & ,i\notin\Lambda \\
\end{cases}$
\end{itemize}We call $\phi$ the phase and $V$ the potential of the interaction.
 For convenience we will frequently omit the boundary symbol from the measure and will
write $\mathbb{E}^{\Lambda}\equiv\mathbb{E}^{\Lambda,\omega}$.
 
\textit{The  Infinite Volume Gibbs Measure.}   The Gibbs measure $\nu$ for the local specification
$\{\mathbb{E}^{\Lambda,\omega}\}_{\Lambda\subset \mathbb{Z},\omega \in \Omega}$ is defined as the  probability measure which solves the Dobrushin-Lanford-Ruelle (DLR) equation
 $$\nu \mathbb{E}^{\Lambda,\star}=\nu $$
for  finite  sets $\Lambda\subset \mathbb{Z}$ (see [P]). For conditions on
the existence and uniqueness of the Gibbs measure see e.g. [B-HK]
and [D]. In this paper we consider  local specifications      for which the Gibbs measure  exists and it is unique. It should be noted that $\{\mathbb{E}^{\Lambda,\omega}\}_{\Lambda\subset\subset \mathbb{Z},\omega \in
\Omega}$ always satisfies the DLR equation, in the sense that
$$\mathbb{E}^{\Lambda,\omega}\mathbb{E}^{M,\ast}=\mathbb{E}^{\Lambda,\omega}$$ for every $M\subset\Lambda$. [P].

\textit{The gradient $\nabla$ for continuous spins systems.}
\noindent
For any subset $\Lambda\subset \mathbb{Z}$ we define the gradient
 $$\left\vert \nabla_{\Lambda} f
\right\vert^q=\sum_{i\in\Lambda}\left\vert \nabla_{i} f
\right\vert^q\text{,\; where \;}\nabla_{i}=\frac{\partial}{\partial x_i} $$
 When $\Lambda=\mathbb{Z}$ we will simply write $\nabla=\nabla_\mathbb{Z}$. We denote
$$\mathbb{E}^{\Lambda,\omega}f=\int f d\mathbb{E}^{\Lambda,\omega}(x_\Lambda)$$
 We can  define the following inequalities
 
\textit{The $q$ Log-Sobolev Inequality  (LS$_q$).}
 We say that the measure $\mathbb{E}^{\Lambda,\omega}$ satisfies the q Log-Sobolev  Inequality for $q\in (1,2]$, if there exists a constant $C_{LS}$   such that for any function $f$, the following  holds
$$\mathbb{E}^{\Lambda,\omega}\left\vert f\right\vert^qlog\frac{\left\vert f\right\vert^q}{\mathbb{E}^{\Lambda,\omega}\left\vert f\right\vert^q}\leq C_{LS}
\mathbb{E}^{\Lambda,\omega}\left\vert \nabla_{\Lambda} f
\right\vert^q$$  
with a constant $ C_{LS}\in(0,\infty)$  uniformly on the set $\Lambda$  and the boundary conditions
$\omega$.
 
\textit{The $q$ Spectral Gap Inequality.} We say that the measure $\mathbb{E}^{\Lambda,\omega}$ satisfies the q Spectral Gap  Inequality for $q\in (1,2]$, if there exists a constant $C_{SG}$   such that for any function $f$, the following  holds
$$\mathbb{E}^{\Lambda,\omega}\left\vert f-\mathbb{E}^{\Lambda,\omega}f \right\vert^q\leq C_{SG}
\mathbb{E}^{\Lambda,\omega}\left\vert \nabla_{\Lambda} f
\right\vert^q$$
with a constant $C_{SG}\in(0,\infty)$  uniformly on the set $\Lambda$  and the boundary conditions
$\omega$.
\begin{rem}\label{rem1.1} We will frequently use the following two well known properties about the
 Log-Sobolev  and the Spectral Gap Inequality. If the probability measure $\mu$ satisfies
the Log-Sobolev Inequality with constant $c$ then it also satisfies the Spectral Gap Inequality with a constant   $\hat c= \frac{4 c}{\log 2}$.
More detailed, in the case where $q=2$ the optimal constant is less or equal to $\frac{c}{2}<\hat c$, while  in the case $1<q<2$ it is less or equal to $\frac{4c}{\log 2}$. The constant $\hat c$ does not depend on the value of  the parameter $q\in(1,2]$.  

Furthermore, if for a family $I$
of sets  $\Lambda_i \subset \mathbb{Z}$,
$dist(\Lambda_i,\Lambda_j)>1 \ , i\neq j$    the measures $\mathbb{E}^{\Lambda_{i},\omega},
i\in I$
satisfy the Log-Sobolev Inequality with constants $c_i,i\in I$,  then the
probability measure $\mathbb{E}^{\{\cup_{i\in I}\Lambda_i\},\omega}$ also satisfies
the (LS) Inequality with constant $c=max_{i\in I} c_i$. The last result is also true for the Spectral Gap Inequality. The
 proofs of these two properties can be found in  [G] and   [G-Z] for $q=2$
 and in [B-Z] for $1<q<2$. \end{rem}

 \section{The Main Result}
 
 We want to extend the Log-Sobolev Inequality from the single-site measure $\mathbb{E}^{\{i\},\omega}$ to the Gibbs measure for the local specification $\{\mathbb{E}^{\Lambda,\omega}\}_{\Lambda\subset\subset \mathbb{Z},\omega \in
\Omega}$  on   the entire one dimensional Lattice.

~
 
 \noindent
\textbf{Hypothesis}
 \noindent We consider four main hypothesis:
 
~
 
\noindent
\textbf{(H0)}: The one dimensional measures $\mathbb{E}^{i,\omega}$ satisfies
the Log-Sobolev-q Inequality with a constant $c$ uniformly with respect to the boundary
conditions $\omega$.

~
 
\noindent
\textbf{(H1)}: The restriction  $\nu_{\Lambda(k)}$ of the Gibbs measure $\nu$
to the $\sigma-$algebra $\Sigma_{\Lambda(k)}$,    $$\Lambda(k)=\{k-2,k-1,k,k+1,k+2\}$$
\ \ \ \   \ \ \ \ \  \ satisfies the Log-Sobolev-q Inequality with a constant $C\in (0,\infty)$.
 
 ~

 \noindent
\textbf{(H2)}:
For some $\epsilon>0$ and $K>0$ $$\nu_{\Lambda(i)} e^{2^{q+2}\epsilon 
V(x_r,x_s)}\leq e^K\text{\; and \;} \nu_{\Lambda(i)} e^{2^{q+2}\epsilon \left\vert\nabla_{r}V(x_{r},x_{s})\right\vert^q}\leq
e^K$$
 for $r,s\in\{i-2,i-1,i,i+1,i+2\}$

\noindent
\textbf{(H3)}:
The coefficients $J_{i,j}$ are such that $\left\vert J_{i,j}\right\vert\in[0,J]$
for some $J<1$ sufficiently
 
\ \ \ \  \ \ \ small.
 
 \begin{rem}\label{rem2.1n} From Hypothesis $(H2)$ and Jensen's inequality
 it follows that   $$\nu e^{\epsilon(\left\vert F(r)\right\vert+\mathbb{E}^{S(r),\omega}\left\vert F(r)\right\vert)^q}\leq
e^K,\text{\; for \;}r=i-2,i-1,i,i+1,i+2$$where the functions $F(r)$ are defined by
$$ F(r)=\begin{cases}\nabla_{r}V(x_{i-1},x_{i})+\nabla_{r}V(x_{i+1},x_{i})& \text{\; for \;}r=i-1,i,i+1  \\
\nabla_{r}V(x_s,x_r)\mathcal{I}_{s\sim r:s\in\{i-3,i+3\}} & \text{\; for \;}r=i-2,i+2  \\
\end{cases}
$$
 and the sets $S(r)$  by 
 $$S(r)=\begin{cases}\{\sim i\} & \text{\; for \;}r=i-1,i,i+1 \\
\{i+3,i+4,...\} & \text{\; for \;}r=i+2 \text{\;and\;}s=i+3\\
\{...,i-4,i-3\} & \text{\; for \;}r=i-2 \text{\;and\;}s=i-3 \\
\end{cases}$$ These bounds will be frequently used through out the paper.
\end{rem} 
\begin{rem}\label{newremark1} Throughout this paper we will consider differentiable functions that satisfy $$\nu\left\vert f\right\vert^q <\infty \text{\; and \;} \nu \left\vert \nabla f\right\vert^q <\infty$$
\end{rem}
 The main theorem  follows.
  \begin{theorem}\label{thm2.1} If hypothesis (H0)-(H3) are satisfied, then the infinite dimensional Gibbs measure $\nu$  for the local specification $\{\mathbb{E}^{\Lambda,\omega}\}_{\Lambda\subset\subset \mathbb{Z},\omega \in
\Omega}$ satisfies the $q$ Log-Sobolev  inequality
$$\nu \left\vert f\right\vert^q log\frac{\left\vert f\right\vert^q}{\nu \left\vert f\right\vert^q}\leq \mathfrak{C} \ \nu \left\vert \nabla f
\right\vert^q$$                              
for some positive constant $\mathfrak{C}$. \end{theorem}
\begin{proof} For the proof of the theorem it is sufficient to consider  $f\geq 0$. This
is an assumption that we will make through all the proofs presented in this paper. We want to extend the Log-Sobolev Inequality from the single-site measure $\mathbb{E}^{\{i\},\omega}$ to the Gibbs measure for the local specification $\{\mathbb{E}^{\Lambda,\omega}\}_{\Lambda\subset\subset \mathbb{Z},\omega \in
\Omega}$  on   the entire one dimensional lattice. To do so, we will follow the iterative method developed by Zegarlinski in [Z1] and [Z2]. Define the following sets
$$\Gamma_0=   \text{\;even integers, \;} \Gamma_1=\mathbb{Z}\smallsetminus\Gamma_0$$
One can notice that $\{dist(i,j)>1, \ \forall i,j \in\Gamma_k,k=0,1\}$, $\Gamma_0\cap\Gamma_1=\emptyset$  and    $\mathbb{Z}=\Gamma_0\cup\Gamma_1$. For convenience we will write $\mathbb{E}^{\Gamma_i}=\mathbb{E}^{\Gamma_i,\omega}$ for $i=0,2$.We will    denote
$$ \mathcal{P}=\mathbb{E}^{\Gamma_1}\mathbb{E}^{\Gamma_{0}}$$
   In order to prove the Log-Sobolev Inequality for the measure $\nu$, we will express the entropy with respect to the measure $\nu$ as the sum of the entropies of the measures   $\mathbb{E}^{\Gamma_0}$  and $\mathbb{E}^{\Gamma_1}$ which are easier to handle.  We can write
 \begin{align}\nonumber \nu (f^q log\frac{f^q}{\nu f^q})=&\nu\mathbb{E}^{\Gamma_0} (f^q log\frac{f^q}{\mathbb{E}^{\Gamma_0} f^q})+\nu\mathbb{E}^{\Gamma_{1}} (\mathbb{E}^{\Gamma_0}f^q log\frac{\mathbb{E}^{\Gamma_0}f^q}{\mathbb{E}^{\Gamma_{1}}\mathbb{E}^{\Gamma_0} f^q})\\ & \label{2.1}+
\nu (\mathbb{E}^{\Gamma_{1}}\mathbb{E}^{\Gamma_{0}}f^q log\mathbb{E}^{\Gamma_{1}}\mathbb{E}^{\Gamma_{0}}f^q)-\nu
 (f^q log \nu f^{q})\end{align}
According to hypothesis (H0), the Log-Sobolev Inequality is satisfied for the single-state measures $\mathbb{E}^{\{j\}}$ and the sets $\Gamma_0$ and $\Gamma_1$   are unions of one dimensional sets of distance greater than the length of the interaction one. Thus, as we mentioned
in Remark~\ref{rem1.1} in the introduction, the (LS) holds for the product measures $\mathbb{E}^{\Gamma_0}$
and $\mathbb{E}^{\Gamma_1}$  with the same constant c. If we use the LS for $\mathbb{E}^{\Gamma_i},i=0,1$  we get
 \begin{align}\nonumber ~\eqref{2.1}\leq & c\nu(\mathbb{E}^{\Gamma_0}\left\vert \nabla_{\Gamma_0} f
\right\vert^q)+c\nu \mathbb{E}^{\Gamma_1}\left\vert \nabla_{\Gamma_1}(\mathbb{E}^{\Gamma_0} f^q)^{\frac{1}{q}}
\right\vert^q\\ &\label{2.2}+\nu (\mathbb{E}^{\Gamma_{1}}\mathbb{E}^{\Gamma_{0}}f^q log\mathbb{E}^{\Gamma_{1}}\mathbb{E}^{\Gamma_{0}}f^q)-\nu (f^q log \nu f^q)\end{align}
For the third term of~\eqref{2.2} we can write
\begin{align}\nonumber\nu (\mathcal{P}f^q log \mathcal{P} f^q)=&\nu \mathbb{E}^{\Gamma_{0}}(\mathcal{P}f^q log \frac{\mathcal{P} f^q}{\mathbb{E}^{\Gamma_{0}}\mathcal{P} f^q})+\nu \mathbb{E}^{\Gamma_{1}}(\mathbb{E}^{\Gamma_{0}}\mathcal{P}f^q log \frac{\mathbb{E}^{\Gamma_{0}}\mathcal{P} f^q}{\mathbb{E}^{\Gamma_{1}}\mathbb{E}^{\Gamma_{0}}\mathcal{P} f^q})\\  & \nonumber
+\nu( \mathbb{E}^{\Gamma_{1}}\mathbb{E}^{\Gamma_{0}}\mathcal{P}f^q log \mathbb{E}^{\Gamma_{1}}\mathbb{E}^{\Gamma_{0}}\mathcal{P}f^q )\end{align}
If we use again the Log-Sobolev Inequality  for the measures $\mathbb{E}^{\Gamma_{i}},i=0,1$
we get
 \begin{equation}\label{2.3}\nu (\mathcal{P}f^qlog \mathcal{P} f^q)\leq c \nu\left\vert \nabla_{\Gamma_0}(\mathcal{P}f^q)^\frac{1}{q}
\right\vert^q+c \nu\left\vert \nabla_{\Gamma_1}(\mathbb{E}^{\Gamma_{0}} \mathcal{P}f^q)^\frac{1}{q}
\right\vert^q+\nu (\mathcal{P}^2f^q log \mathcal{P}^2f^q) \end{equation}
If we work similarly for the last term $\nu (\mathcal{P}^2f^q log \mathcal{P}^2f^q)$  of~\eqref{2.3} and  inductively for any term  $\nu (\mathcal{P}^kf^q log \mathcal{P}^kf^q)$, then after $n$ steps~\eqref{2.2} and~\eqref{2.3} will give
\begin{align}\nonumber\nu (f^q log\frac{f^q}{\nu f^q})\leq &\nu (\mathcal{P}^n f^q \log\mathcal{P}^n f^q)-\nu (f^qlog \nu f^q)+c\nu \left\vert \nabla_{\Gamma_0}f\right\vert^q\\  &
\label{2.4}+c \sum_{k=1}^{n-1} \nu \left\vert \nabla_{\Gamma_0}( \mathcal{P}^kf^q)^\frac{1}{q}
\right\vert^q+c \sum_{k=0}^{n-1} \nu\left\vert \nabla_{\Gamma_1}(\mathbb{E}^{\Gamma_{0}} \mathcal{P}^kf^q)^\frac{1}{q}
\right\vert^q \end{align}
 In order to calculate the fourth and fifth term on the right-hand side of~\eqref{2.4} we will use the following proposition
\begin{proposition}\label{prp2.2}Suppose  that hypothesis (H0)-(H3) are satisfied.
Then the following bound holds
\begin{equation} \label{2.5}\nu\left\vert \nabla_{\Gamma_i}(\mathbb{E}^{\Gamma_{j}} \vert f\vert^q)^\frac{1}{q}
\right\vert^q\leq C_1\nu\left\vert \nabla_{\Gamma_i}f\right\vert^q+C_2\nu\left\vert \nabla_{\Gamma_j}f\right\vert^q\end{equation}
for $\{i,j\}=\{0,1\}$ and  constants $C_1\in (0,\infty)$ and $0<C_2<1$.\end{proposition}
The proof of Proposition~\ref{prp2.2} will be the subject of Section 4.
 If we apply inductively relationship~\eqref{2.5}
  k times to  the  fourth and the fifth term of~\eqref{2.4} we obtain
 \begin{equation} \label{2.6} \nu\left\vert \nabla_{\Gamma_0}(\mathcal{P}^kf^q)^\frac{1}{q}
\right\vert^q \leq C_2^{2k-1}C_1\nu\left\vert \nabla_{\Gamma_1} f
\right\vert^q+C_2^{2k}\nu\left\vert \nabla_{\Gamma_0} f
\right\vert^q\end{equation}
     and
\begin{equation}\label{2.7}\nu\left\vert \nabla_{\Gamma_1}(\mathbb{E}^{\Gamma_{0}}\mathcal{P}^kf^q)^\frac{1}{q}
\right\vert^q \leq C_2^{2k}C_1\nu\left\vert \nabla_{\Gamma_1} f
\right\vert^q+C_2^{2k+1}\nu\left\vert \nabla_{\Gamma_0} f
\right\vert^q\end{equation}
If we plug~\eqref{2.6} and~\eqref{2.7} in~\eqref{2.4} we get 
\begin{align}\nonumber\nu (f^qlog\frac{f^q}{\nu f^q})\leq&\nu (\mathcal{P}^n f^q log \mathcal{P}^n f^q)-\nu (f^qlog \nu f^q)\\ \nonumber
 &+c(\sum_{k=0}^{n-1}C_2^{2k-1})C_1\nu\left\vert \nabla_{\Gamma_1} f
\right\vert^q+c(\sum_{k=0}^{n-1}C_2^{2k})\nu\left\vert \nabla_{\Gamma_0} f
\right\vert^q\\ &
 \label{2.8}+c(\sum_{k=0}^{n-1}C_2^{2k})C_1\nu\left\vert \nabla_{\Gamma_1} f
\right\vert^q+c(\sum_{k=0}^{n-1}C_2^{2k+1})\nu\left\vert \nabla_{\Gamma_0} f
\right\vert^q\end{align}
 If we take the limit of $n$  to infinity in~\eqref{2.8} the first two terms on
 the right hand side cancel with each other,  as explained on the proposition
 bellow.
 \begin{proposition}\label{prp2.3} Under hypothesis (H0)-(H3), $\mathcal{P}^nf$ converges
 $\nu$-almost everywhere to $\nu f$.  
 \end{proposition}
 The proof of this proposition will be presented in Section 3.
So, taking the limit of $n$  to infinity in~\eqref{2.8} leads to  $$\nu (\vert f\vert^qlog\frac{\vert f\vert^q}{\nu \vert f\vert^q})\leq cA\left(\frac{C_1}{C_2}+C_2+C_1\right)\nu\left\vert \nabla_{\Gamma_1} f
\right\vert^q+cA\nu\left\vert \nabla_{\Gamma_0} f
\right\vert^q$$
 where $A=lim_{n\rightarrow\infty}\sum_{k=0}^{n-1}C_2^{2k}<\infty$  for  $C_2<1$, and the theorem follows for a constant $C=max\{cA\left(\frac{C_1}{C_2}+C_2+C_1\right),cA\}$ \end{proof}
 \section{Proof of Proposition \ref{prp2.3}. }
   Before proving Proposition~\ref{prp2.3} we will present three useful lemmata. These
lemmata will also be   used in the next section 4 where Proposition~\ref{prp2.2} is
proved. 

In the case of quadratic interactions $V(x,y)=(x-y)^2$ one can calculate $$\mathbb{E}^{i,\omega}\left( f^2(\nabla_j V(x_i-x_j)-\mathbb{E}^{i,\omega}\nabla_j V(x_i-x_j))^2 \right)$$  (see [B-H] and [H]) with the use of the Deuschel-Stroock relative entropy inequality (see [D-S]) and the Herbst argument (see [L] and [H]). Herbst's arguement states that if a probability measure $\mu$ satisfies the LS2 inequality and a function $F$ is Lipschitz continues with $\Vert F \Vert_{Lips}\leq 1$ and such that $\mu (F)=0$, then for some small $\epsilon$ we have  $$\mu e^{\epsilon F^2}<\infty$$
For $\mu=\mathbb{E}^{i,\omega}$ and $F=\frac{\nabla_j V(x_i-x_j)-\mathbb{E}^{i,\omega}\nabla_j V(x_i-x_j)}{2}$ we then obtain  $$\mathbb{E}^{i,\omega} e^{\frac{\epsilon}{4} (\nabla_j V(x_i-x_j)-\mathbb{E}^{i,\omega}\nabla_j V(x_i-x_j))^2}<\infty$$uniformly on the boundary conditions $\omega$, because of hypothesis $(H0)$. In the more general case however of non quadratic interactions that we examine in this work, the Herbst argument cannot be applied. In this and next sections we show how one can bound exponential quantities like the last one with the use of the projection of the infinite dimensional Gibbs measure and hypothesis (H1) and (H2). 

For every probability measure $\mu$, we define the correlation function $$\mu(f;g)\equiv\mu(fg)-\mu(f)\mu(g)$$If for the set $M(k)=\mathbb{Z}\smallsetminus\Lambda(k)$ and  $h_k:=f-\mathbb{E}^{\{ \sim k\}}f$ we define$$Q(u,k)\equiv\nu_{\Lambda(u)}\left\vert \nabla_{\Lambda(u)}\left( \mathbb{E}^{M(u)}
\vert h_k\vert ^q \right)^{\frac{1}{q}}\right\vert^q$$
 then the following lemma presents an estimate
for the correlation function, in terms of $Q(k,k)$.

\begin{lemma}\label{lem3.1}
 For any functions  $u$ localised in $\Lambda(k)$  for which $\nu_{\Lambda(k)} e^{2^{q}\epsilon \vert u\vert^q}<\infty$ the following inequalities are satisfied

~

\noindent (a)
under hypothesis  (H1)\begin{align*}\nu \left\vert\mathbb{E}^{k-1}\mathbb{E}^{k+1}(f;u)\right\vert^q\leq&\frac{C}{\epsilon }Q(k,k)+  \frac{1}{\epsilon }\left(log\nu_{\Lambda(k)} e^{\epsilon \vert u-\mathbb{E}^{k-1}\mathbb{E}^{k+1}u\vert^q}\right)\nu\left\vert f-\mathbb{E}^{k-1}\mathbb{E}^{k+1}f\right \vert^{q}
\end{align*}

\noindent (b) under hypothesis (H0) and (H1)$$
\nu \left\vert\mathbb{E}^{k-1}\mathbb{E}^{k+1}(f;u)\right\vert^q\leq\frac{C}{\epsilon }Q(k,k)+\frac{\hat c}{\epsilon }\left(log\nu_{\Lambda(k)} e^{\epsilon \vert u-\mathbb{E}^{k-1}\mathbb{E}^{k+1}u\vert^q}\right)\sum_{i=k-1,k+1}\nu\left\vert \nabla_ i f\right\vert^q
$$
where $\hat c =\frac{4 c}{\log 2}$.\end{lemma}
\begin{proof} From  the definition of the correlation function we can write
 \begin{align}\nonumber\nu\left\vert\mathbb{E}^{k-1}\mathbb{E}^{k+1}
(f;u)\right\vert^q= & \nu\left\vert\mathbb{E}^{k-1}\mathbb{E}^{k+1}((f-\mathbb{E}^{k-1}\mathbb{E}^{k+1}f)(u-\mathbb{E}^{k-1}\mathbb{E}^{k+1}u))\right\vert^q
\\ \leq & \nu\mathbb{E}^{k-1}\mathbb{E}^{k+1}\left(\vert f-\mathbb{E}^{k-1}\mathbb{E}^{k+1}
f\vert^q\vert u-\mathbb{E}^{k-1}\mathbb{E}^{k+1}u\vert^q\right) \nonumber\\ = & \label{3.1}\nu\left(\vert f-\mathbb{E}^{k-1}\mathbb{E}^{k+1}
f\vert^q\vert u-\mathbb{E}^{k-1}\mathbb{E}^{k+1}u\vert^q\right)\end{align}
 where above we first used the Jensen's Inequality and then the fact that
 the Gibbs measure
$\nu$  satisfies the DLR equation. Because the function $u$ is localised
 in $\Lambda(k)$ and the measure $\mathbb{E}^{\{k-1,k+1\},\omega}=\mathbb{E}^{k-1}\mathbb{E}^{k+1}$ has boundary in $\{k-2,k,k+2\}\subset\Lambda(k)$, we have that $u-\mathbb{E}^{k-1}\mathbb{E}^{k+1}u$ is also localised in $\Lambda(k)$ and so for $M(k)$ being the complementary
of $\Lambda(k)$ we can write
 \begin{align}\nonumber\nu(\vert f-\mathbb{E}^{k-1}\mathbb{E}^{k+1}
f\vert^{q}&\vert u-\mathbb{E}^{k-1}\mathbb{E}^{k+1}u\vert^q)=&\\ &\label{3.2}\nu_{\Lambda(k)}\left(\left(\mathbb{E}^{M(k)}\vert f-\mathbb{E}^{k-1}\mathbb{E}^{k+1}f\vert^{q}\right)\vert u-\mathbb{E}^{k-1}\mathbb{E}^{k+1}u\vert^q\right)\end{align}
 On the right hand side of~\eqref{3.2}  we can use the following entropic inequality
(see [D-S])
\begin{equation}\label{3.3}\forall t>0, \ \mu(uy)\leq\frac{1}{t}log\left(\mu(e^{tu})\right)+\frac{1}{t}\mu(y\log y)\end{equation}
 for any probability measure $\mu$  and $y\geq 0$, $\mu  y=1$. Then from~\eqref{3.1}
 and~\eqref{3.2}  we will obtain  
   \begin{align}\nonumber \nu&\left\vert\mathbb{E}^{k-1}\mathbb{E}^{k+1}
(f;u)\right\vert^q\leq& \\ & \frac{1}{\epsilon }\nu_{\Lambda(k)}\mathbb{E}^{M(k)}
\vert f-\mathbb{E}^{k-1}\mathbb{E}^{k+1}
f\vert^{q}\log\frac{\mathbb{E}^{M(k)}
\vert f-\mathbb{E}^{k-1}\mathbb{E}^{k+1}
f\vert^{q}}{\nu_{\Lambda(k)}\mathbb{E}^{M(k)}
\vert f-\mathbb{E}^{k-1}\mathbb{E}^{k+1}
f\vert^{q}}\nonumber \\ &\label{3.4}+\frac{1}{\epsilon}\left(log\nu_{\Lambda(k)}e^{\epsilon
\vert u-\mathbb{E}^{k-1}\mathbb{E}^{k+1}
u\vert^q}\right)\nu_{\Lambda(k)}\mathbb{E}^{M(k)}\vert f-\mathbb{E}^{k-1}\mathbb{E}^{k+1}f\vert^{q}\end{align}
The first term on the right hand side of~\eqref{3.4} can be bounded from hypothesis (H1) by the Log-Sobolev inequality for        
 $\nu_{\Lambda(k)}$  
   \begin{align} \nonumber \nu_{\Lambda(k)}\mathbb{E}^{M(k)}
\vert f-\mathbb{E}^{k-1}\mathbb{E}^{k+1}
f\vert^{q}&\log\frac{\mathbb{E}^{M(k)}
\vert f-\mathbb{E}^{k-1}\mathbb{E}^{k+1}
f\vert^{q}}{\nu_{\Lambda(k)}\mathbb{E}^{M(k)}
\vert f-\mathbb{E}^{k-1}\mathbb{E}^{k+1}
f\vert^{q}}\\ \leq &   \label{3.5}C\nu_{\Lambda(k)}\left\vert \nabla_{\Lambda(k)}(\mathbb{E}^{M(k)}
\vert f-\mathbb{E}^{k-1}\mathbb{E}^{k+1}
f\vert^{q})^{\frac{1}{q}}\right \vert ^q=CQ(k,k)\end{align}
 Using~\eqref{3.4} and~\eqref{3.5}  we get
\begin{align}\label{3.5+1equation}
\nu \left\vert\mathbb{E}^{k-1}\mathbb{E}^{k+1}(f;u)\right\vert^q\leq\frac{C}{\epsilon }Q(k,k)+\frac{1}{\epsilon }\left(log\nu e^{\epsilon \vert u-\mathbb{E}^{k-1}\mathbb{E}^{k+1}u\vert^q}\right)\nu\left\vert f-\mathbb{E}^{k-1}\mathbb{E}^{k+1}f\right \vert^{q}
\end{align}which proves (a). If we assume hypothesis (H0), then we can bound the  second term on the right hand side of (\ref{3.5+1equation}) from the $SG_q$ for the measures $\mathbb{E}^{k-1},\mathbb{E}^{k+1}$ from hypothesis (H0) and
the product property for the $SG_q$ (Remark \ref{rem1.1}), to
obtain
\begin{align} \nu\vert f-\mathbb{E}^{k-1}\mathbb{E}^{k+1}f\vert^{q}=&\nu\mathbb{E}^{k-1}\mathbb{E}^{k+1}\vert f-\mathbb{E}^{k-1}\mathbb{E}^{k+1}f\vert^{q}
 \leq &\label{3.6}\hat c\sum_{i=k-1,k+1}\nu\left\vert \nabla_ i f\right\vert^q\end{align}
where $\hat c =\frac{4 c}{\log 2}$. Using (\ref{3.5+1equation}) and ~\eqref{3.6} we finally get
(b)$$
\nu \left\vert\mathbb{E}^{k-1}\mathbb{E}^{k+1}(f;u)\right\vert^q\leq\frac{C}{\epsilon }Q(k,k)+\frac{\hat c}{\epsilon }\left(log\nu e^{\epsilon \vert u-\mathbb{E}^{k-1}\mathbb{E}^{k+1}u\vert^q}\right)\sum_{i=k-1,k+1}\nu\left\vert \nabla_ i f\right\vert^q
$$\end{proof}
 The following lemma gives an explicit bound for the quantity $Q(k,k)$.

\begin{lemma}\label{lem3.2}Suppose that  hypothesis (H0)-(H3) are satisfied. Then
 \begin{align*} Q(k,k)\leq&D\sum_{r=k-2}^{k+2}\nu\left\vert \nabla_{r}f
\right\vert^q\\  & +D\sum_{ n=0 }^{\infty} J^{(n+1)(q-1)}\sum_{r=0}^3\left(\nu\left\vert \nabla_{k+3+4n+r} f
\right\vert^q+\nu\left\vert \nabla_{k-3-4n-r} f
\right\vert^q\right)\end{align*}for some positive constant $D$.\end{lemma}
 The proof of this lemma will be the subject of Section 5.

\begin{lemma} \label{lem3.3}Suppose  that hypothesis (H0)-(H3) are satisfied. Then for $\{i,j\}=\{0,1\}$
$$\nu\left\vert \nabla_{\Gamma_i}(\mathbb{E}^{\Gamma_{j}}f)
\right\vert^q \leq D_1\mathcal{\nu}\left\vert \nabla_{\Gamma_i} f
\right\vert^q+D_2\nu\left\vert \nabla_{\Gamma_j} f
\right\vert^q$$ 
  holds for constants $D_1\in (0,\infty)$  and $0<D_2<1$.  \end{lemma}

 \begin{proof} Assume $i=1,j=0$. We have
 \begin{equation}\label{3.7}\nu\left\vert \nabla_{\Gamma_1}(\mathbb{E}^{\Gamma_{0}}f)
\right\vert^q=\sum_{i\in \Gamma_1} \nu\left\vert \nabla_{i}(\mathbb{E}^{\Gamma_{0}}f)
\right\vert^q\leq\sum_{i\in \Gamma_1} \nu\left\vert \nabla_{i}(\mathbb{E}^{i-1}\mathbb{E}^{i+1}f)
\right\vert^q\end{equation}
 If we denote   $\rho_i= \frac{e^{-H(x_{i-1})}e^{-H(x_{i+1})}}{\int e^{-H(x_{i-1})}dx_i\int e^{-H(x_{i+1})}dx_i}$ the density of the measure $\mathbb{E}^{i-1}\mathbb{E}^{i+1}$ we can then write
 \begin{align}\nonumber\nu&\left\vert\nabla_{i}(\mathbb{E}^{i-1}\mathbb{E}^{i+1}f)\right\vert^{q}=
 \nu\left\vert\nabla_{i}(\int \int \rho_{i}  f dx_{i-1}dx_{i+1})\right\vert^{q}\leq
   &\\ &  \label{3.8}2^{q-1}\nu\left\vert\int \int (\nabla_{i}f) \rho_{i} dx_{i-1}dx_{i+1}\right\vert^{q}+\ 2^{q-1}\nu\left\vert\int \int f(\nabla_{i}\rho_{i} )dx_{i-1}dx_{i+1}\right\vert^{q}\leq
\\ & \label{3.9}c_{1}\nu\left\vert\mathbb{E}^{i-1}\mathbb{E}^{i+1}(\nabla_{i}f)\right\vert^{q}+\ c_{1}J^q\nu\left\vert\mathbb{E}^{i-1}\mathbb{E}^{i+1}(f; \nabla_{i}V(x_{i-1},x_{i})+\nabla_{i}V(x_{i+1},x_{i}))\right\vert^{q}
\end{align}
where in~\eqref{3.9} we used hypothesis (H3) to bound the coefficients $J_{i,j}$
and
 we have denoted $c_1=2^{4q}$. If we apply the H\"older  Inequality to the first term of~\eqref{3.9} and Lemma~\ref{lem3.1} (b) to the
second term, we obtain
\begin{equation}\label{3.10}\nu\left\vert\nabla_{i}(\mathbb{E}^{i-1}\mathbb{E}^{i+1}f)\right\vert^q
\leq c_{1}\nu\left\vert\nabla_{i}f\right\vert^q+\frac{J^{q}c_1C}{\epsilon }Q(i,i)+\frac{J^{q}\hat cc_{1}K}{\epsilon }\sum_{k=i-1,i+1}\nu\left\vert \nabla_ k f
\right\vert^q \end{equation}
where the constant $K$ as in hypothesis $(H2)$. From~\eqref{3.7} and~\eqref{3.10} we have
\begin{align*}
\nu  \left\vert \nabla_{\Gamma_1}(\mathbb{E}^{\Gamma_0}f)
\right\vert^q\leq  c_1\nu\left\vert\nabla_{ \Gamma_1}f\right\vert^q+\frac{J^q c_1C}{\epsilon}\sum_{i\in \Gamma_1}Q(i,i)+
 \frac{J^q\hat cc_1K}{\epsilon }\sum_{i\in \Gamma_1}\sum_{k=i-1,i+1}\nu\left\vert \nabla_ k
f
\right\vert^q\end{align*}
 If we use Lemma~\ref{lem3.2} to replace $Q(k,k)$ in the above expression we get
\begin{align*}\nonumber\nu  \left\vert \nabla_{\Gamma_1}(\mathbb{E}^{\Gamma_0}f)\right\vert^q&\leq c_1\nu\left\vert\nabla_{ \Gamma_1}f\right\vert^q+\frac{J^q\hat c2c_1K}{\epsilon}\nu\left\vert\nabla_{ \Gamma_0}f\right\vert^q+\frac{J^q c_1DC}{\epsilon }\sum_{i\in \Gamma_1}\sum_{r=i-2}^{i+2}\nu\left\vert \nabla_r f\right\vert^q+ \\ & \frac{J^q c_1DC}{\epsilon }\sum_{i\in \Gamma_1}\sum_{n=0 }^{\infty} J^{(n+1)(q-1)}\sum_{r=0}^3\left(\nu\left\vert \nabla_{i+3+4n+r} f
\right\vert^q+\nu\left\vert \nabla_{i-3-4n-r} f
\right\vert^q\right)\end{align*}
for constant $D>0$ as in Lemma~\ref{lem3.2}. For coefficients $J_{i,j}$ sufficiently small such that $J<1$ in (H3) we finally obtain
 
~
 
$\nu  \left\vert \nabla_{\Gamma_1}(\mathbb{E}^{\Gamma_0}f)
\right\vert^q\leq J^{q}\left(\frac{2\hat cc_{1}K}{\epsilon}+\frac{2c_{1}CD}{\epsilon }+2D\frac{c_{1}C}{\epsilon }\frac{J^{(q-1)}}{1-J^{(q-1)}}\right)\nu\left\vert \nabla_ { \Gamma_0}
f
\right\vert^q$
 
~
 
$ \   \ \ \  \ \  \   \  \  \   \   \   \   \   \   \   \   \   \   \   \   \   \   \    \   \   \  \   \  +\left(  c_1q+\frac{J^{q}c_{1}C}{\epsilon }3D+D\frac{2J^{q}c_{1}C}{\epsilon }\frac{J^{(q-1)}}{1-J^{(q-1)}}\right)\nu\left\vert\nabla_{ \Gamma_1}f\right\vert^q$

 ~

\noindent
and the lemma follows for $J$ sufficiently small such that
$$D_2=J^{q}\left(\frac{\hat cc_{1}K}{\epsilon}2+\frac{2c_1CD}{\epsilon }+D\frac{2c_1C}{\epsilon }\frac{J^{(q-1)}}{1-J^{(q-1)}}\right)<1$$  \end{proof}

~

 Now we can prove
 Proposition~\ref{prp2.3}.

 ~

 \noindent
\textit{\textbf{Proof of Proposition~\ref{prp2.3}}.} Following [G-Z] we will  show that in   $L_1(\nu)$ we have   $lim_{n\rightarrow \infty}\mathcal{P}^n=\nu$.
  For $i\not=j$ we have that
 \begin{align}\nonumber\nu\vert\mathbb{E}^{\Gamma_{j}} f- \mathbb{E}^{\Gamma_{i}}\mathbb{E}^{\Gamma_{j}} f\vert^q&=\nu\mathbb{E}^{\Gamma_{i}}\vert\mathbb{E}^{\Gamma_{j}} f- \mathbb{E}^{\Gamma_{i}}\mathbb{E}^{\Gamma_{j}} f\vert^q\\ &
 \label{3.11}\leq  \hat c\nu\left\vert \nabla_{\Gamma_i}(\mathbb{E}^{\Gamma_{j}} f
)\right\vert^q\end{align}          
The last inequality due to the  fact that both  the measures  $\mathbb{E}^{\Gamma_{0}}$  and  $\mathbb{E}^{\Gamma_{1}}$
satisfy the Log-Sobolev Inequality and the Spectral Gap inequality with constants  independently of the boundary conditions. If we use Lemma \ref{lem3.3} we get
$$\nu\vert\mathbb{E}^{\Gamma_{j}} f- \mathbb{E}^{\Gamma_{i}}\mathbb{E}^{\Gamma_{j}} f\vert^q\leq  \hat cD_1\nu \vert \nabla_{\Gamma_i}f\vert^q+\hat cD_2\nu \vert \nabla_{\Gamma_j}f\vert^q$$  
  From the last inequality we obtain that for any   $n\in \mathbb{N}$, 
 \begin{align*}\nu\vert \mathcal{P}^{n}f- \mathbb{E}^{\Gamma_0}\mathcal{P}^{n} f\vert^q &\leq  \hat cD_1\nu \vert \nabla_{\Gamma_0}(\mathbb{E}^{\Gamma_0}\mathcal{P}^{n-1} f)\vert^q+\hat cD_2\nu \vert \nabla_{\Gamma_1}(\mathbb{E}^{\Gamma_0}\mathcal{P}^{n-1} f)\vert^q\\&=\hat cD_2\nu \vert \nabla_{\Gamma_1}(\mathbb{E}^{\Gamma_0}\mathcal{P}^{n-1} f)\vert^q\end{align*}
 If we use Lemma \ref{lem3.3} to bound the last expression we have the following 
\begin{equation} \label{BorelCant1} \nu\vert \mathcal{P}^{n}f- \mathbb{E}^{\Gamma_0}\mathcal{P}^{n} f\vert^q \leq \hat cD_2^{n}\left(D_1\nu\left\vert \nabla_{\Gamma_1} f
\right\vert^q+D_2\nu\left\vert \nabla_{\Gamma_0} f
\right\vert^q\right)\end{equation} Similarly we obtain 
 \begin{equation} \label{BorelCant1+2} \nu\vert\mathbb{E}^{\Gamma_0} \mathcal{P}^{n}f- \mathcal{P}^{n+1} f\vert^q\leq  \hat cD_2^{n}\left(D_1\nu\left\vert \nabla_{\Gamma_1} f
\right\vert^q+D_2\nu\left\vert \nabla_{\Gamma_0} f
\right\vert^q\right)\end{equation}

Consider the sequence $\{\mathcal{Q}^n\}_{n\in \mathbb{N}}$ defined as 
$$\mathcal{Q}^{n}f=\begin{cases}\mathcal{P}^{\frac{n}{2}}f & \text{ if \ } n \text{ \ even} \\
\mathbb{E}^{\Gamma_0}\mathcal{P}^{\frac{n-1}{2}}f & \text{ if \ } n \text{ \ odd} \\
\end{cases}$$
for every $n\in\mathbb{N}$. Hence, if we define the sets $$A_n=\{\vert \mathcal{Q}^nf- \mathcal{Q}^{n+1} f\vert \geq (\frac{1}{2})^n\}$$ we obtain 
\begin{equation*}\nu (A_n)=\nu\left( \{\vert \mathcal{Q}^nf- \mathcal{Q}^{n+1} f\vert \geq (\frac{1}{2})^n\} \right)\leq 2^{qn}\nu\vert \mathcal{Q}^nf- \mathcal{Q}^{n+1} f\vert^q  \end{equation*}by Chebyshev inequality. If we use (\ref{BorelCant1}) and (\ref{BorelCant1+2}) to bound the last we have 
\begin{equation*}\nu (A_n)\leq (2^{q}D_2^\frac{1}{2} )^{n}\hat c\left(D_1\nu\left\vert \nabla_{\Gamma_1} f
\right\vert^q+D_2\nu\left\vert \nabla_{\Gamma_0} f
\right\vert^q\right)
\end{equation*}
We can choose $J$ sufficiently small such that $2^{q}D_2^\frac{1}{2}<\frac{1}{2}$ in which case we get that \begin{equation*}\sum_{n=0}^\infty\nu (A_n)\leq\left(\sum_{n=0}^\infty(\frac{1}{2})^{n} \right)\hat c\left(D_1\nu\left\vert \nabla_{\Gamma_1} f
\right\vert^q+D_2\nu\left\vert \nabla_{\Gamma_0} f
\right\vert^q\right)<\infty
\end{equation*}   From the  Borel-Cantelli lemma, only finite number of the sets $A_n$ can occur, which implies that  the sequence
 $$\{\mathcal{Q}^{n} f\}_{n \in \mathbb{N}}$$ 
 is a Cauchy sequence and that it converges $\nu-$almost surely. Say $$\mathcal{Q}^nf\rightarrow \theta (f) \   \   \   \   \   \nu-\text{a.e.}$$ We will first show that $\theta (f)$  is a constant, i.e. it does not depend on variables on $\Gamma_0$ or $\Gamma_1$. To show that, first notice that  $\mathcal{Q}^n(f)$ is a function on $\Gamma_1$ and $\Gamma_0$ when $n$  is odd and even respectively, which implies that the limits  $$\theta_{o}(f)=\lim_{n \text{\ odd}, n\rightarrow \infty}\mathcal{Q}^nf \text{ \ and \ }\theta_e(f)=\lim_{n \text{\ even}, n\rightarrow \infty}\mathcal{Q}^nf$$ do not depend on variables on $\Gamma_0$ and $\Gamma_1$ respectively.    Since both the  subsequences $\{\mathcal{Q}^nf\}_{n\text{\ even}}$ and $\{\mathcal{Q}^nf\}_{n\text{\ odd}}$  converge to $\theta (f)$ $\nu-$a.e.  we have that $$\theta_{o}(f)=\theta(f)=\theta_{e}(f)$$
 which implies that $\theta (f)$ is a constant. From that we obtain that \begin{equation}\label{BorelCant6} \nu \left( \theta (f) \right)=\theta(f)
\end{equation} 
Since the sequence $\{ \mathcal{Q}^n f\}_{n \in \mathbb{N}}$ 
 converges $\nu-$almost, the same holds for  the sequence $\{ \mathcal{Q}^n f-\nu \mathcal{Q}^n f\}_{n \in \mathbb{N}}$. We have  $$\lim_{n\rightarrow \infty}(\mathcal{Q}^n f-\nu \mathcal{Q}^n f)=\theta (f)-\nu\left( \theta(f) \right) =\theta (f)-\theta (f)=0$$
where above we used (\ref{BorelCant6}). On the other side, we also have 
 \begin{equation}\label{BorelCant7}\lim_{n\rightarrow \infty}( \mathcal{Q}^n f-\nu \mathcal{Q}^n f)=\lim_{n\rightarrow \infty}(\mathcal{Q}^n f-\nu f)=\theta (f)-\nu (f)
\end{equation} 
From (\ref{BorelCant6}) and (\ref{BorelCant7}) we get that 
$$\theta (f)=\nu (f)$$
We finally get $$\lim_{n\rightarrow \infty}\mathcal{P}^n f=\lim_{n \text{\ even}, n\rightarrow \infty}\mathcal{Q}^nf=\nu f,  \  \   \nu \  \  a.e. $$
\qed  
 
  \section{Proof of Proposition~\ref{prp2.2}}
 Before we prove Proposition~\ref{prp2.2} we present some useful lemmata. First we
 define \begin{equation}\label{defineW_k}W_k=\nabla_kV(x_{k},x_{k-1})+\nabla_kV(x_{k},x_{k+1}) \text{\; and \;}
U_k=\left\vert W_{k}\right\vert^q+\mathbb{E}^{\{\sim k\}}\left\vert W_{k}\right\vert^q\end{equation}
where  $\{\sim k \}\equiv\{j:j\sim k\}=\{k-1,k+1\}$.
\begin{lemma}\label{lem4.1} The following inequality holds  
\begin{align*}
\mathbb{E}^{\{\sim k\}}( f^q;  W_{k})
 \leq   c_{0}\left(\mathbb{E}^{\{\sim k\}}\vert f\vert^q\right)^{\frac{1}{p}}\left( \mathbb{E}^{\{\sim k \}}(\vert f-\mathbb{E}^{\{\sim k\}}f\vert^qU_k) \right)^\frac{1}{q}\end{align*} for some constant $c_0$ uniformly on the boundary conditions and  $\frac{1}{q}+\frac{1}{p}=1$. \end{lemma}

\begin{proof} We can  write
\begin{align}\label{4.1}\mathbb{E}^{\{\sim k\}}(f^q;W_{k})=
\frac{1}{2}\mathbb{E}^{\{\sim k\}}\otimes\mathbb{\tilde
E}^{\{\sim k\}}\left((f^q-\tilde f^{q})(W_{k}-\tilde W_k  )\right)\end{align}
 where $\mathbb{\tilde
E}^{\{\sim k\}}$ is an isomorphic copy of $\mathbb{E}^{\{\sim k\}}$. If we define the function $F$  to be  $F(s)=sf+(1-s)\tilde f$ then
 \begin{align} & ~\eqref{4.1}=\frac{1}{2}\mathbb{E}^{\{\sim k\}}\otimes\mathbb{\tilde
E}^{\{\sim k\}}\left(\left(\int_0^1ds\frac{d}{ds}
 F(s)^{q}\right)(W_{k}-\tilde W_k) \right)\nonumber \\ \nonumber& \ \ \
 \ \ \ \ = \frac{1}{2}\mathbb{E}^{\{\sim k\}}\otimes\mathbb{\tilde
E}^{\{\sim k\}}\left(\left(\int_0^1dsq
 F(s)^{q-1}\frac{d}{ds}F(s)\right)(W_{k}-\tilde W_k ) \right) \\ &\ \ \
 \ \ \
 \  =\frac{1}{2}\mathbb{E}^{\{\sim k\}}\otimes\mathbb{\tilde
E}^{\{\sim k\}}\left(\left(q\int_0^1dsF(s)^{q-1}(f-\tilde f)\right)(W_{k}-\tilde W_k  )\right) \nonumber\end{align}
If we use the Holder inequality
for the conjugate numbers $p$ and $q$, then the last quantity can be bounded
by
 
~
$ \ \ \ \ \ \ \ \ \ \frac{q}{2}\left\{\mathbb{E}^{\{\sim k\}}\otimes\mathbb{\tilde
E}^{\{\sim k\}}\left( \int_0^1dsF(s)^{q-1}  \right)^p\right\}^{\frac{1}{p}}\times$
 
\begin{equation} \label{4.2}\ \ \ \ \ \ \ \ \ \ \ \ \ \ \ \ \ \ \ \ \ \ \ \ \ \  \left\{\mathbb{E}^{\{\sim k\}}\otimes\mathbb{\tilde
E}^{\{\sim k\}}\left\vert(f-\tilde f)(W_{k}-\tilde W_k)  \right\vert^q\right\}^{\frac{1}{q}}\end{equation}
For the first term in the above product, by Jensen's Inequality  and $\frac{1}{q}+\frac{1}{p}=1$, we obtain  
\begin{align}  & \left\{ \mathbb{E}^{\{\sim k\}}\otimes\mathbb{\tilde
E}^{\{\sim k\}} \left(\int_0^1dsF(s)^{q-1}\right)^p\right\} ^{\frac{1}{p}}\nonumber \\ &\nonumber \ \ \ \ \ \ \ \ \ \leq \left\{\mathbb{E}^{\{\sim k\}}\otimes\mathbb{\tilde
E}^{\{\sim k\}}\int_0^1dsF(s)^{q}\right\}^{\frac{1}{p}}=  \left(\int_0^1ds\mathbb{E}^{\{\sim k\}}\otimes\mathbb{\tilde
E}^{\{\sim k\}}F(s)^{q}\right)^{\frac{1}{p}}  \\ & \label{4.3}
\ \ \ \ \ \ \ \ \ \leq \left( 2^q\int_0^1ds\mathbb{E}^{\{\sim k\}}\otimes\mathbb{\tilde
E}^{\{\sim k\}} \left(sf^{q}+(1-s)\tilde
 f^{q} \right)  \right)^{\frac{1}{p}}= 2^{\frac{q}{p}}(\mathbb{E}^{\{\sim k\}} f^{q} )^\frac{1}{p}\end{align} 
 If we plug~\eqref{4.3} into~\eqref{4.2} we finally get
 \begin{align*}\mathbb{E}^{\{\sim k\}}(f^q;W_{k})\leq &\frac{2^{\frac{q}{p}}q}{2}(\mathbb{E}^{\{\sim k\}}f^{q} )  ^\frac{1}{p}\left\{\mathbb{E}^{\{\sim k\}}\otimes\mathbb{\tilde
E}^{\{\sim k\}}\left(\vert f-\tilde f\vert \vert W_{k}-\tilde W_k \vert \right)^q\right\}^{\frac{1}{q}}\\
\leq& 2^62^{\frac{q}{p}}q(\mathbb{E}^{\{\sim k\}} f^{q} )  ^\frac{1}{p}\left\{\mathbb{E}^{\{\sim k\}} \left(\vert f-\mathbb{E}^{\{\sim k\}} f\vert ^q(\left\vert W_{k}\right\vert^q+\mathbb{E}^{\{\sim k\}}\left\vert W_{k}\right\vert^q)  \right)\right\}^{\frac{1}{q}}
\end{align*}
 The lemma follows for constant  $c_0=2^62^{\frac{q}{p}}q$.\end{proof}
  Define now the quantity $$ A(k)=\nu\left(\mathbb{E}^{\{\sim k\}}  \vert f\vert^q\right)^{-\frac{q}{p}}\left\vert\mathbb{E}^{\{\sim k\}} (\vert f\vert^q;W_{k})\right\vert^q$$The
 next lemma presents an estimate of $A(k)$ involving $Q(k,k)$.
 \begin{lemma} \label{lem4.2} Suppose that that hypothesis (H0)-(H2)  are satisfied. Then
  $$A(k)\leq\frac{c^{q}_{0}C}{\epsilon  } Q(k,k)+\frac{c^{q}_{0}\hat cK}{ \epsilon}\sum_{i=k-1,k+1}\nu\left\vert \nabla_{i} f\right\vert^q$$
 where
  the constants $\epsilon$ and $K$ are as in hypothesis (H2).\end{lemma}
\begin{proof} We can initially bound $A(k)$ with the use of Lemma~\ref{lem4.1}
 \begin{align}\nonumber A(k)=&\nu\left(\mathbb{E}^{\{\sim k\}}f^q\right)^{-\frac{q}{p}}\left\vert\mathbb{E}^{\{\sim k\}}(f^q;W_{k})\right\vert^q\leq c^{q}_{0}\nu \mathbb{E}^{k-1}\mathbb{E}^{k+1}(\vert f-\mathbb{E}^{k-1}\mathbb{E}^{k+1}f\vert^qU_k)\\ = & \label{4.4}c^{q}_{0}\nu_{\Lambda(k)}\left((\mathbb{E}^{M(k)}\vert f-\mathbb{E}^{\{\sim k\}}f\vert^q)U_k\right) \end{align}
 because $U_k$ is localized in $\Lambda(k)$. If we use the entropy inequality~\eqref{3.3} and hypothesis (H1) for
 $\nu_{\Lambda(k)}$  as well as $(H2)$, as we did in Lemma
\ref{lem3.1}, then for $K$ as in $(H2)$, we can bound~\eqref{4.4} by
\begin{align*}~\eqref{4.4}\leq &\frac{c^{q}_{0}C}{\epsilon  } Q(k,k)+ \frac{c^{q}_{0}K}{ \epsilon}\nu_{\Lambda(k)}\mathbb{E}^{\{\sim k\}}\vert f-\mathbb{E}^{\{\sim k\}}f\vert^q \\ \leq& \frac{c^{q}_{0}C}{\epsilon  } Q(k,k)+\frac{c^{q}_{0}\hat cK}{ \epsilon} \sum_{i=k-1,k+1}\nu\left\vert \nabla_{i} f\right\vert^q\end{align*}where above we used that $\mathbb{E}^{\{\sim
k\}}=\mathbb{E}^{k-1}\mathbb{E}^{k+1}$
satisfies the $SG_q$ with constant $\hat c$ uniformly on the boundary conditions, by hypothesis
(H0) and Remark \ref{rem1.1}.\end{proof}
\begin{lemma} \label{lem4.3}The following inequality holds
  $$\nu\left\vert \nabla_{i}(\mathbb{E}^{i-1}\mathbb{E}^{i+1}\vert f\vert^q)^\frac{1}{q}
\right\vert^q \leq c_1\nu\left\vert \nabla_i f
\right\vert^q+\frac{J^qc_1}{q^q}  A(i) $$
 \end{lemma}
\begin{proof}  We have
   \begin{align}\nonumber\nu\left\vert \nabla_{i}(\mathbb{E}^{i-1}\mathbb{E}^{i+1}f^q)^\frac{1}{q}
\right\vert^q=&\nu\left\vert \frac{1}{q}(\mathbb{E}^{i-1}\mathbb{E}^{i+1}f^q)^{\frac{1}{q}-1}
\nabla_{i}(\mathbb{E}^{i-1}\mathbb{E}^{i+1}f^q)\right\vert^q \\ =&\label{4.5}\frac{1}{q^{q}}\nu  (\mathbb{E}^{i-1}\mathbb{E}^{i+1}f^q)^{-\frac{q}{p}}
\left\vert\nabla_{i}(\mathbb{E}^{i-1}\mathbb{E}^{i+1}f^q)\right\vert^q\end{align} But from relationship~\eqref{3.8} of Lemma~\ref{3.3}, for $\rho_i$ being the density of
$\mathbb{E}^{\{\sim i\}}$ we have

~

$\left\vert\nabla_{i}(\mathbb{E}^{i-1}\mathbb{E}^{i+1}f^q)\right\vert^{q}=\left\vert\nabla_{i}(\int \int \rho_{i}  f^q dx_{i-1}dx_{i+1})\right\vert^{q}\leq $
\begin{equation}  \   \ \ \  \label{4.6}2^{2q-2}\left\vert\int \int \nabla_{i}(f^q) \rho_{i} dx_{i-1}dx_{i+1}\right\vert^{q}+ 2^{2q-2}\left\vert\int \int f^q(\nabla_{i}\rho_{i} )dx_{i-1}dx_{i+1}\right\vert^{q}\end{equation}
 For the second term in~\eqref{4.6} we have
\begin{equation}\label{4.7} \left\vert\int \int f^q(\nabla_{i}\rho_{i} )dx_{i-1}dx_{i+1}\right\vert^{q}\leq J^{q}\left\vert\mathbb{E}^{i-1}\mathbb{E}^{i+1}(f^q; \nabla_{i}V(x_{i-1},x_{i})+\nabla_{i}V(x_{i+1},x_{i}))\right\vert^q\end{equation}
While for the   first term of~\eqref{4.6} the following bound holds
\begin{align}\nonumber \left\vert\int \int \nabla_{i}(f^q) \rho_{i} dx_{i-1}dx_{i+1}\right\vert^{q}&=q^{q}\left\vert\mathbb{E}^{i-1}\mathbb{E}^{i+1}(f^{q-1}( \nabla_{i}f)) \right\vert^{q}\\ \nonumber & \leq
q^{q}\left(  \mathbb{E}^{i-1}\mathbb{E}^{i+1} f^{(q-1)p} \right)^{\frac{q}{p}}\left(\mathbb{E}^{i-1}\mathbb{E}^{i+1}\left\vert\nabla_{i}f\right\vert^q\right)\\
&\label{4.8}=q^{q}\left(  \mathbb{E}^{i-1}\mathbb{E}^{i+1} f^{q} \right)^{\frac{q}{p}}\left(\mathbb{E}^{i-1}\mathbb{E}^{i+1}\left\vert\nabla_{i}f\right\vert^q\right)\end{align}
 where above we used the H\"older inequality and that $p$ is the conjugate of $q$.
If we plug~\eqref{4.7} and~\eqref{4.8} in~\eqref{4.6} we get
\begin{align*}\left\vert\nabla_{i}(\mathbb{E}^{i-1}\mathbb{E}^{i+1}f^q)\right\vert^q\leq & 2^{2q-2}q^{q}\left(  \mathbb{E}^{i-1}\mathbb{E}^{i+1} f^{q} \right)^{\frac{q}{p}}\left(\mathbb{E}^{i-1}\mathbb{E}^{i+1}\left\vert\nabla_{i}f\right\vert^q\right)\\
 &+2^{2q-2}J^{q}\left\vert\mathbb{E}^{i-1}\mathbb{E}^{i+1}(f^q; \nabla_{i}V(x_{i-1},x_{i})+\nabla_{i}V(x_{i+1},x_{i}))\right\vert
^q\end{align*}
 From the last relationship and~\eqref{4.5} the lemma follows.  \end{proof}   Now we can prove Proposition~\ref{prp2.2}.

~

 \noindent\textit{\textbf{Proof of Proposition~\ref{prp2.2}}.} We  have
\begin{align}\nonumber \nu\left\vert \nabla_{\Gamma_1}(\mathbb{E}^{\Gamma_{0}} f^q)^\frac{1}{q}
\right\vert^q=&\sum_{i\in \Gamma_1} \nu\left\vert \nabla_{i}(\mathbb{E}^{\Gamma_{0}} f^q)^\frac{1}{q}
\right\vert^q\leq\sum_{i\in \Gamma_1}\nu\left\vert \nabla_{i}(\mathbb{E}^{\{ \sim i\}} f^q)^\frac{1}{q}
\right\vert^q \\ \leq &  \sum_{i\in \Gamma_1}c_1\nu\left\vert \nabla_i f
\right\vert^q+\frac{J^qc_1}{q^q} \sum_{i\in \Gamma_1} A(i)\nonumber\end{align}
 where the last inequality is due to Lemma~\ref{lem4.3}.
If we use Lemma~\ref{lem4.2} to bound $A(i)$ we get
\begin{align}\nonumber \nu\left\vert \nabla_{\Gamma_1}(\mathbb{E}^{\Gamma_{0}} f^q)^\frac{1}{q}
\right\vert^q\leq&\sum_{i\in \Gamma_1}c_1\nu\left\vert \nabla_i f
\right\vert^q+\frac{c^{q}_{0}\hat c K}{ \epsilon}\frac{J^qc_1}{q^q} \sum_{i\in \Gamma_1}\sum_{r=i-1,i+1}\nu\left\vert \nabla_{r} f\right\vert^q\\ &+\frac{J^qc_1}{q^q} \frac{c^{q}_{0}C}{\epsilon  }\sum_{i\in \Gamma_1} Q(i,i)\nonumber\end{align}
 Furthermore, if we use Lemma~\ref{lem3.2} to bound $Q(i,i)$ we obtain
  \begin{align}\nu\left\vert \nabla_{\Gamma_1}(\mathbb{E}^{\Gamma_0}
  f^q)^\frac{1}{q}
\right\vert^q\leq & \nonumber\sum_{i\in \Gamma_1}c_1\nu\left\vert \nabla_i f
\right\vert^q+\frac{c^{q}_{0}\hat cK}{ \epsilon}\frac{J^qc_1}{q^q} \sum_{i\in \Gamma_1}\sum_{r=i-1,i+1}\nu\left\vert \nabla_{r} f\right\vert^q
 \\ & \nonumber +\frac{J^qc_1}{q^q} \frac{c^{q}_{0}CD}{\epsilon  }\sum_{i\in \Gamma_1}   \sum_{r=k-2}^{k+2}\nu\left\vert \nabla_{r}f
\right\vert^q
\\ & +\frac{J^qc_1}{q^q} \frac{c^{q}_{0}CD}{\epsilon  }\sum_{i\in \Gamma_1}\sum_{ n=0 }^{\infty} J^{(n+1)(q-1)}\sum_{r=0}^3\nu\left\vert \nabla_{i+3+4n+r} f
\right\vert^q\nonumber \\ & \label{4.9} +\frac{J^qc_1}{q^q} \frac{c^{q}_{0}CD}{\epsilon  }\sum_{i\in \Gamma_1}\sum_{ n=0 }^{\infty} J^{(n+1)(q-1)}\sum_{r=0}^3\nu\left\vert \nabla_{i-3-4n-r} f
\right\vert^q\end{align}
 If we set $R=c_1+\frac{c_1}{q^q} (\frac{c^{q}_{0}CD}{\epsilon  }+\frac{c^{q}_{0}\hat cK}{ \epsilon})$ and we choose $J<1$, relationship~\eqref{4.9}
 gives  
$$\nu\left\vert \nabla_{\Gamma_1}(\mathbb{E}^{\Gamma_{0}} f^q)^\frac{1}{q}
\right\vert^q\leq  (R+RJ^{q}4+\frac{R8J^{q}}{1-J^{q-1}})\nu\left\vert \nabla_{ \Gamma_1} f
\right\vert^q+RJ^{q}(4+\frac{8}{1-J^{q-1}})\nu\left\vert \nabla_{ \Gamma_0} f
\right\vert^q$$
  For $J$ sufficiently small (H3) such that $RJ^{q}(4+\frac{8}{1-J^{q-1}})<1$
the lemma follows for constants
 
~
 
 $C_1=R+RJ^{q}4+\frac{R8J^{q}}{1-J^{q-1}}$ and $C_2=RJ^{q}(4+\frac{8}{1-J^{q-1}})<1$.\qed
\section{Proof of Lemma~\ref{lem3.2}}

 This section is dedicated in the proof of Lemma~\ref{lem3.2} under the assumptions (H0)-(H3). We begin by showing the weaker result of Lemma \ref{50} under the weaker assumptions (H1)-(H3). \begin{lemma} \label{50}Suppose that  hypothesis (H1)-(H3) are satisfied. Then
\begin{align*} Q(k,k)\leq & J^{q}S\nu\left\vert f-\mathbb{E}^{k-1}\mathbb{E}^{k+1}f\right\vert^q+
S\sum_{r=k-2}^{k+2}\nu\left\vert \nabla_{r}
f
\right\vert^q\\&+S\sum_{ n=0 }^{\infty} J^{(n+1)(q-1)}\sum_{r=0}^3\left(\nu\left\vert \nabla_{k+3+4n+r} f
\right\vert^q+\nu\left\vert \nabla_{k-3-4n-r} f
\right\vert^q\right)\end{align*}for some positive constant $S$.\end{lemma}
  Lemma \ref{lem3.2} follows for some constant $D>0$  directly from the last lemma and the Spectral Gap inequality implied from (H0).
The remaining of this section is dedicated to the proof of Lemma \ref{50}. At first we prove some
lemmata.
To start, for any $k \in \mathbb{Z}$,  we define the sets $M_s(k)$ for  $s=k-3,k+3$ as  \begin{equation}\label{Ms(k)}
 M_s(k)=\begin{cases}\{j\in\mathbb{Z}:j\geq k+3\}=\{k+3,k+4,...\}\ & \text{if \  }  s=k+3\\
\{j\in\mathbb{Z}:j\leq\ k-3\}=\{...,k-4,k-3\}\ & \text{if \  }  s=k-3 \\
\end{cases}\end{equation} 
 \begin{rem}\label{remarkMs(k)} Since 
$\Lambda(k)=\{k-2,k-1,k,k+1,k+2\}$ and $M(k)=\mathbb{Z}\smallsetminus\Lambda(k)$, with the use of the definition (\ref{Ms(k)}) we can write \begin{equation*}M(k)= \{j\in\mathbb{Z}:j\leq\ k-3\}\cup\{j\in\mathbb{Z}:j\geq k+3\}=M_{k-3}(k)\cup M_{k+3}(k) \end{equation*}
Since the sets $M_{k-3}(k)$ and $M_{k+3}(k)$ are disjoint we obtain that  $\mathbb{E}^{M(k)}$ is a product measure, and for every function $f$ we can write\begin{equation}\label{productmeasurein1dim}\mathbb{E}^{M(k)}f=\mathbb{E}^{M_{k-3}(k)}\otimes\mathbb{E}^{M_{k+3}(k)}f\end{equation}
Accordingly, for  functions, say $f_{k-3}$ and $f_{k+3}$, that   depend on variables $x_i$ with   $i\notin\ M_{k+3}(k)$ and $i\notin\ M_{k-3}(k)$ respectively, we  obtain $$\mathbb{E}^{M(k)}f_{k-3}=\mathbb{E}^{M_{k-3}(k)}f_{k-3}$$ and $$\mathbb{E}^{M(k)}f_{k+3}=\mathbb{E}^{M_{k+3}(k)}f_{k+3}$$ For instance, for $r=k-2,k+2$ and $ s\in \{k-3,k+3\}:s\sim r$, that is for  the couples $(r,s)=(k-2,k-3)$ and $(r,s)=(k+2,k+3)$, we have
\begin{equation}\label{5.2}\mathbb{E}^{M(k)}\nabla_rV(x_{s},x_{r})=\mathbb{E}^{M_s(k)}\nabla_rV(x_{s},x_{r})\end{equation}
\end{rem}
\begin{rem}\label{remarkMs(k)2} Consider couples   $(r,s)$ that take the values $ (k-2,k-3)$ and $(k+2,k+3)$. We then have that $\nabla_rV(x_{s},x_{r})$  is localised in $\Lambda(k-4)$ when $(r,s)=(k-2,k-3)$ and in $\Lambda(k+4)$ when $(r,s)=(k+2,k+3)$. Furthermore, from Remark \ref{remarkMs(k)}, for $(r,s)=(k-2,k-3)$ we get that $$\mathbb{E}^{M_{k-3}(k)}\nabla_{k-2}V(x_{k-3},x_{k-2})=\mathbb{E}^{\{...,k-4,k-3\}}\nabla_{k-2}V(x_{k-3},x_{k-2})$$ is localised in $\Lambda(k-4)$, while  for $(r,s)=(k+2,k+3)$ we get that $$\mathbb{E}^{M_{k+3}(k)}\nabla_{k+2}V(x_{k+3},x_{k+2})=\mathbb{E}^{\{k+3,k+4,...\}}\nabla_{k+2}V(x_{k+3},x_{k+2})$$ is localised in $\Lambda(k+4)$. 
 So, if we set $$Y_{s}(x_{s},x_{r})=\vert\nabla_rV(x_{s},x_{r})-\mathbb{E}^{M_s(k)}\nabla_rV(x_{s},x_{r})\vert$$
 we then have that $Y_{k+3}(x_{k+2},x_{k+3})$ and $Y_{k-3}(x_{k-2},x_{k-3})$ are localised in $\Lambda(k+4)$ and $\Lambda(k-4)$ respectively.
Thus, we have
\begin{align*}\nu(f^qY_{k+3}^q(x_{k+3},x_{k+2}))=\nu_{\Lambda(k+4)}\left((\mathbb{E}^{M(k+4)}f^q)Y_{k+3}^q(x_{k+3},x_{k+2})\right)
 \end{align*}
 \begin{align*}\nu(f^qY_{k-3}^q(x_{k-3},x_{k-2}))=\nu_{\Lambda(k-4)}\left((\mathbb{E}^{M(k-4)}f^q)Y_{k-3}^q(x_{k-3},x_{k-2})\right)
 \end{align*}
 If we combine the last two together we can write
\begin{align*}\nu(f^qY_{s}^q(x_{s},x_{r}))=\nu_{\Lambda(t)}\left((\mathbb{E}^{M(t)}f^q)Y^q_{s}(x_{s},x_{r})\mathcal{I}_{t\in
 \{k-4,k+4\}\cap M_s(k)}\right)
 \end{align*} for $(r,s)\in \{(k+2,k+3),(k-2,k-3)\}$.\end{rem}
 \begin{lemma} \label{lem5.1}Suppose conditions  (H1) and (H2) are satisfied.
Then for $r=k-2,k+2$ and $ s\in \{k-3,k+3\}:s\sim r$ the following inequality
is true
\begin{align*}
\nu_{\Lambda(k)}\left(\mathbb{E}^{M(k)}
\vert f\vert^q\right)^{-\frac{q}{p}}&\left\vert\mathbb{E}^{M(k)}(\vert
f\vert^q;\nabla_rV(x_{s},x_{r}))\right\vert^q \leq \\ &
 \frac{C}{ \epsilon }\nu_{\Lambda(t)}\left\vert \nabla_{\Lambda(t)}
(\mathbb{E}^{M(t)}\vert f\vert^q )^{\frac{1}{q}}\right\vert^q\mathcal{I}_{t\in
 \{k-4,k+4\}\cap M_s(k)}+\frac{K}{\epsilon}\nu \vert f\vert^q\end{align*}
where $\mathcal{I}_A$ 
denotes the
characteristic function of a set $A$ and the set  $M_s(k)$ as in (\ref{Ms(k)}).\end{lemma}
\begin{proof} For any two function $f$ and $g$ the covariance with respect to a measure $\mu$ can be computed  as bellow  
\begin{align*}\mu (f;g)&=\mu\left((f-\mu f)(g-\mu g)\right)=\mu\left (f(g-\mu g)\right)-\mu\left(\mu f(g-\mu g)\right)\\  &=\mu\left(f(g-\mu g)\right)-(\mu f)\mu\left(g-\mu g\right)=\mu\left (f(g-\mu g)\right)\end{align*}  Using this expression we can write
\begin{align}\label{5.1} \mathbb{E}^{M(k)}(f^q;\nabla_rV(x_{s},x_{r}))= \mathbb{E}^{M(k)}(
f^q(\nabla_rV(x_{s},x_{r})-\mathbb{E}^{M(k)}\nabla_rV(x_{s},x_{r})))\end{align}
 If we use    \eqref{5.2} from Remark \ref{remarkMs(k)}, ~\eqref{5.1} becomes
\begin{equation}\label{5.3}\mathbb{E}^{M(k)}(f^q;\nabla_rV(x_{s},x_{r}))=\mathbb{E}^{M(k)}(f^q(\nabla_rV(x_{s},x_{r})-\mathbb{E}^{M_s(k)}\nabla_rV(x_{s},x_{r})))
\end{equation}
\noindent If we set $$Y_{s}(x_{s},x_{r})=\vert\nabla_rV(x_{s},x_{r})-\mathbb{E}^{M_s(k)}\nabla_rV(x_{s},x_{r})\vert$$
 then for~\eqref{5.3} we can write   
 \begin{align}\nonumber\left\vert \mathbb{E}^{M(k)}(f^q;\nabla_rV(x_{s},x_{r})) \right\vert&\leq\mathbb{E}^{M(k)}(f^{q-1+1}Y_{s}(x_{s},x_{r}))
\\ &  \leq\left(\mathbb{E}^{M(k)}f^{(q-1)p}\right)^\frac{1}{p}\left (\mathbb{E}^{M(k)}(f^qY^q_{s}(x_{s},x_{r}))\right
)^{\frac{1}{q}}\nonumber\\
 &\label{5.4}=\left(\mathbb{E}^{M(k)}f^{q}\right)^\frac{1}{p}\left(\mathbb{E}^{M(k)}(f^qY^q_{s}(x_{s},x_{r}))\right
 )^{\frac{1}{q}}\end{align}
  where above we used the H\"older inequality and that  $\frac{1}{p}+\frac{1}{q}=1$.
So, for $s=k+3,k-3$ from relationship~\eqref{5.4} we obtain
\begin{align*}\nonumber\nu_{\Lambda(k)}\left(\mathbb{E}^{M(k)}f^q\right)^{-\frac{q}{p}}\left\vert\mathbb{E}^{M(k)}(f^q;\nabla_rV(x_{s},x_{r}))\right
\vert^{q}& \leq\nu_{\Lambda(k)}\mathbb{E}^{M(k)}(f^qY_{s}^q(x_{s},x_{r}))\\  &  =\nu(f^qY_{s}^q(x_{s},x_{r}))\end{align*}
  If we combine the last inequality together   with  Remark \ref{remarkMs(k)2}  we finally obtain
\begin{align}\nonumber\nu_{\Lambda(k)}\left(\mathbb{E}^{M(k)}f^q\right)^{-\frac{q}{p}}&\left\vert\mathbb{E}^{M(k)}(f^q;\nabla_rV(x_{s},x_{r}))\right
\vert^{q} \leq\\ &
 \label{5.5}\nu_{\Lambda(t)}\left((\mathbb{E}^{M(t)}f^q)Y^q_{s}(x_{s},x_{r})\mathcal{I}_{t\in
 \{k-4,k+4\}\cap M_s(k)}\right)\end{align} If in~\eqref{5.5} we use the Entropy Inequality and the $LS_q$ for $\nu_{\Lambda(s)}$ from hypothesis
 (H1) and (H2), we get
  \begin{align*} \nu_{\Lambda(k)}\left(\mathbb{E}^{M(k)}
f^q\right)^{-\frac{q}{p}} &\left(\mathbb{E}^{M(k)}(
f^q;\nabla_rV(x_{s},x_{r}))\right)^q \leq  \\ & 
 \frac{C}{ \epsilon }\nu_{\Lambda(t)}\left\vert \nabla_{\Lambda(t)}
(\mathbb{E}^{M(t)}f^q )^{\frac{1}{q}}\right\vert^q\mathcal{I}_{t\in
 \{k-4,k+4\}\cap M_s(k)}+\frac{K}{\epsilon}\nu f^q \end{align*}
and $ s\in \{k-3,k+3\}:s\sim r$ and
$K$ and $\epsilon$ as in hypothesis (H2). \end{proof}

~

 \begin{lemma}  \label{lem5.2-1}Suppose  $P$ and $G$ are positive functions with domain on $\mathbb{N}$ such that
for constants  $J,K'>0$ 
\begin{equation}\label{l6-5.6}P(4)\leq G(4)+J^qK'P(8)  \end{equation}
and for $n=4k$ for  $k\in\mathbb{N}\cap [2,\infty)$ 
 \begin{equation}\label{l6-5.7}P(n)\leq G(n)+J^qK'P(n-4)+ J^qK'P(n+4)   \end{equation} Then for $J$ sufficiently small such that   \begin{equation}\label{l6-5.8}J\leq1 \ \text{\;
and \;}JK'+J^qK'J^{q-1}\leqslant 1  \end{equation} the following inequality holds   
 \begin{align}P(4n)\nonumber \leq&\frac{1}{1-J^qK'J^{q-1}}\sum_{m=0}^{n-2} J^{mq-m}G(4n-4m)+J^{(n-1)q-(n-1)}G(4) \\ &\label{l6-5.9}+ J^{q-1}P(4n+4) \end{align}
for any $n\in \mathbb{N},n\geq2$ . \end{lemma}
\begin{proof}  In order to show~\eqref{l6-5.9} we will work inductively.

~

\noindent
Step 1: The base case  of the induction (n=2). 

We  prove~\eqref{l6-5.9} for  $n=2 $.  For $k=8$ in (\ref{l6-5.7}) we have
$$ P(8)\leq G(8)+J^qK'P(12)+ J^qK'P(4) $$
If we bound $P(4)$ in the above inequality by (\ref{l6-5.6}) we obtain 
$$P(8) \leq G(8)+J^qK'P(12)+ J^qK'G(4)+(J^{q}K')^2P(8)\Rightarrow$$
\begin{equation}\label{5.10}P(8)\leq \frac{1}{1-(J^{q}K')^2}G(8)+ \frac{J^qK'}{1-(J^{q}K')^2}G(4)+\frac{J^qK'}{1-(J^{q}K')^2}P(12)\end{equation}
 For $J$   satisfying properties~\eqref{l6-5.8}, we have $JK'+J^qK'J^{q-1}\leq 1 $ and $JK'<1$ which implies
 \begin{align} \label{5.11}JK'+(J^{q}K')^2&\leq  1 \Rightarrow
  \frac{J^qK'}{1-(J^{q}K')^2}\leq{J^{q-1}}\end{align}
  From~\eqref{5.10} and~\eqref{5.11} we have
\begin{align*}P(8)& \leq\frac{1}{1-(J^{q}K')^2}G(8)+ J^{q-1}G(4)+J^{q-1}P(12)\\&\leq\frac{1}{1-J^{q}K'J^{q-1}}G(8)+ J^{q-1}G(4)+J^{q-1}P(12)\end{align*} because of~\eqref{l6-5.8}. This proves~\eqref{l6-5.9} for $n=2$.
 
 ~

 \noindent 
 Step 2: The induction step.
 Suppose the inequality~\eqref{l6-5.9} is true for $n=k$.  Then we will show it is also true for
 $n=k+1$.

If we use~\eqref{l6-5.7} for $n=4k+4$ we have
\begin{equation}\label{5.12}P(4k+4)\leq G(4k+4)+J^qK'P(4k)+ J^qK'P(4k+8)  \end{equation}
  If we use~\eqref{l6-5.9} for $n=k$ to bound $P(4k)$ in~\eqref{5.12} we get
 \begin{align*}P(4k+4)\leq&  G(4k+4)+ \frac{J^qK'}{1-J^qK'J^{q-1}}\sum_{m=0}^{k-2} J^{mq-m}G(4k-4m)\\&
 +J^qK'J^{(k-1)q-(k-1)}G(4)+J^qK'J^{q-1}P(4k+4)+ J^qK'P(4k+8) \end{align*}
 This implies
\begin{align}\nonumber P(4k+4)\leq & \frac{1}{1-J^qK'J^{q-1}}G(4k+4)+ \frac{J^qK'}{1-J^qK'J^{q-1}}\sum_{m=0}^{k-2} \frac{J^{mq-m}}{1-J^qK'J^{q-1}}G(4k-4m)\\ &+
\label{5.13} \frac{J^qK'J^{(k-1)q-(k-1)}}{1-J^qK'J^{q-1}}G(4)+ \frac{J^qK'}{1-J^qK'J^{q-1}}P(4k+8)\end{align}
If we use condition~\eqref{l6-5.8} for $J$, ~\eqref{5.13} becomes
 
~

$P(4k+4)\leq \frac{1}{1-J^qK'J^{q-1}}\sum_{m=0}^{k-1} J^{mq-m}G(4k+4-4m)+J^{kq-k}G(4)$
 
~
 
$\ \ \ \ \ \ \ \ \ \ \ \ \   \ \ \ \  +J^{q-1}P(4k+8)$
 
~

\noindent which proves ~\eqref{l6-5.9} for $n=k+1$. This finishes the proof of ~\eqref{l6-5.9}.
 
 \end{proof}
  \begin{lemma}  \label{lem5.2}Suppose  $P$ and $G$ are positive functions with domain on $\mathbb{N}$ such that
for constants  $J,K'>0$ one has\begin{equation}\label{newlemma5.4}\sup_{n\in\mathbb{N}}P(n)<\infty
\end{equation} as well as (\ref{l6-5.6}) and (\ref{l6-5.7}) for $n=4k$ for  $k\in\mathbb{N}\cap [2,\infty)$. Then for $J$ sufficiently small such that   (\ref{l6-5.8}) is true, the following inequality holds $$P(4)\leq\frac{1}{1-J^{2q-2}}\sum_{n=0}^{+\infty} J^{nq-n}G(4n+4)$$\end{lemma}
\begin{proof}

~

\noindent
 We can use relationship ~\eqref{l6-5.9} from Lemma \ref{lem5.2-1}  to  prove the lemma.
 We first replace the bound of $P(8)$ from ~\eqref{l6-5.9} in ~\eqref{l6-5.6}, to obtain
 \begin{align*}P(4)&\leq G(4)+J^qK'\frac{1}{1-J^qK'J^{q-1}}G(8)+J^qK'J^{q-1}G(4)+J^qK'J^{q-1}P(12) \\ & \leq   (1+J^qK'J^{q-1})G(4)+J^{2q-2}G(8)+JK'J^{2q-2}P(12) \end{align*}
 where at the last inequality we used ~\eqref{l6-5.8}. If we now bound in the
above expression   $P(12)$
from ~\eqref{l6-5.9}, then $P(16)$ from ~\eqref{l6-5.9} and so on,  we will finally obtain
 \begin{align*}P(4)\leq& (1+J^{q}K'\sum_{n=0}^{+\infty}J^{(2n+1)q-(2n+1)} )G(4) \\ &+
  \frac{J^qK'}{1-J^qK'J^{q-1}}\sum_{n=1}^{+\infty} J^{(n-1)q-(n-1)}(\sum_{s=0}^{+\infty}J^{2sq-2s})G(4n+4)\\
 = &
    (1+\frac{J^{2q-1}K'}{1-J^{2q-2}} )G(4)+ \frac{J^qK'}{1-J^qK'J^{q-1}}\frac{1}{1-J^{2q-2}}\sum_{n=1}^{+\infty} J^{(n-1)q-(n-1)}G(4n+4)\end{align*}
  where above we used that $J<1$,  as well as that $$\lim_{n\rightarrow\infty}J^{nq-n}P(8+4n)=0$$ since   (\ref{newlemma5.4}) is true. Furthermore, if we use again ~\eqref{l6-5.8} we then get   $$P(4)\leq \frac{1}{1-J^{2q-2}}\sum_{n=0}^{+\infty} J^{nq-n}G(4n+4) $$  \end{proof}

~

 The next lemma presents a bound for $$Q(u,k)=\nu_{\Lambda(u)}\left\vert \nabla_{\Lambda(u)}\left( \mathbb{E}^{M(u)}
\vert h_k\vert ^q \right)^{\frac{1}{q}}\right\vert^q$$ in terms of $Q(t,k)\mathcal{I}_{dist(u,t)=4}$.
 
 \begin{lemma} \label{lem5.3}Under hypothesis (H1) and (H2) the following bound
 for $Q(u,k)$ holds
\begin{align*}Q(u, k)\leq&\nu_{\Lambda(u)}\left\vert \nabla_{u}
h_k
\right\vert^q+\sum_{r= u-1,u+1}\nu_{\Lambda(u)}\left\vert \nabla_{r}
h_k
\right\vert^q+\frac{J^qc_{1}2cK}{\epsilon}\nu\left\vert h_{k}\right\vert^q  \\  &+c_{1}\sum_{r= u-2,u+2}\nu_{\Lambda(u)}\left\vert \nabla_{r} h_{k}
\right\vert^q+ \frac{J^qc_{1}C}{ \epsilon}\sum_{dist(u,t)=4}Q(t,k)\end{align*}
where $h_{k}=f-\mathbb{E}^{\{\sim k\}}f$.\end{lemma}
 
\begin{proof} We have
\begin{align}Q(u,k)=&\nonumber\nu_{\Lambda(u)}\left\vert \nabla_{\Lambda(u)}\left( \mathbb{E}^{M(u)}
\vert h_k\vert^q \right)^{\frac{1}{q}}\right\vert^q=\nu_{\Lambda(u)}\left\vert \nabla_{u} (\mathbb{E}^{M(u)}
\vert h_k\vert^q)^{\frac{1}{q}}
\right\vert^q\\ &
\label{5.14}+\sum_{r= u-1,u+1}\nu_{\Lambda(u)}\left\vert \nabla_{r} (\mathbb{E}^{M(u)}
\vert h_k\vert^q)^{\frac{1}{q}}
\right\vert^q+\sum_{r= u-2,u+2}\nu_{\Lambda(u)}\left\vert \nabla_{r} (\mathbb{E}^{M(u)}
\vert h_k\vert^q)^{\frac{1}{q}}
\right\vert^q \end{align}
  For $r=u-1,u,u+1$
\begin{equation}\label{5.15}\nu_{\Lambda(u)}\left\vert \nabla_{r} (\mathbb{E}^{M(u)}
\vert h_k\vert^q)^{\frac{1}{q}}
\right\vert^q\leq\nu_{\Lambda(u)}\left\vert \nabla_r
h_k
\right\vert^q\end{equation}
   For $r=u-2,u+2$
 \begin{align}\nonumber\nu_{\Lambda(u)}&\left\vert \nabla_{r} (\mathbb{E}^{M(u)}
\vert h_k\vert^q)^{\frac{1}{q}}
\right\vert^q\leq c_1\nu_{\Lambda(u)}\left\vert \nabla_{r} h_{k}
\right\vert^q+\\ &
 \label{5.16}\frac{J^qc_1}{q^q}\nu_{\Lambda(u)}\left(\mathbb{E}^{M(u)}\vert h_k\vert^{q}\right)^{-\frac{q}{p}}\left(\mathbb{E}^{M(u)}(\vert h_k\vert^q;\nabla_rV(x_{r},x_{s}))\right)^q
 \mathcal{I}_{s\in \{u-3,u+3\}:s\sim r}
\end{align} 
 For $s\in \{u-3,u+3\}:s\sim r$, if we use Lemma~\ref{lem5.1}  we obtain
\begin{align}\nonumber\nu_{\Lambda(u)}&\left(\mathbb{E}^{M(u)}\vert h_k\vert^{q}\right)^{-\frac{q}{p}}\left\vert\mathbb{E}^{M(u)}(\vert h_k\vert^q;V'(x_{r},x_{s}))\right\vert^q
 \mathcal{I}_{s\in \{u-1,u+1\}:s\sim r}\leq\\ &\frac{C}{ \epsilon }\nu_{\Lambda(t)}\left\vert \nabla_{\Lambda(t)}
(\mathbb{E}^{M(t)}\vert h_k\vert^q )^{\frac{1}{q}}\right\vert^q\mathcal{I}_{(s,t)=(u+1,u+4)\cup(u-1,u-4):s
\sim r}+\frac{K}{\epsilon}\nu \vert h_k\vert^q
\label{5.17}\end{align} 
 From~\eqref{5.16} and~\eqref{5.17} we get
 \begin{align}\nonumber\nu_{\Lambda(u)}&\left\vert \nabla_{r} (\mathbb{E}^{M(u)}
\vert h_k\vert^q)^{\frac{1}{q}}
\right\vert^q\leq c_{1}\nu_{\Lambda(u)}\left\vert \nabla_{r} h_{k}
\right\vert^q\\&
\nonumber+\frac{J^qc_{1}C}{ \epsilon }Q(t,k)\mathcal{I}_{(s,t)=(u+1,u+4)\cup(u-1,u-4):s
\sim r}\mathcal{I}_{s\in \{u-1,u+1\}:s\sim r}\\ &\label{5.18}+\frac{J^qc_{1}K}{\epsilon}\nu\left\vert h_k\right\vert^q\mathcal{I}_{s\in \{u-1,u+1\}:s\sim r}\end{align}
To summarise, if we plug~\eqref{5.15} and~\eqref{5.18} in~\eqref{5.14} we finally obtain
\begin{align*}Q(u,k)\leq&\frac{J^qc_1C}{ \epsilon}\sum_{dist(u,t)=4}Q(t,k)+\frac{J^qc_12K}{\epsilon}\nu\left\vert f-\mathbb{E}^{\{\sim k\}}f\right\vert^q+\nu\left\vert \nabla_u
h_k
\right\vert^q\\ &
+\sum_{r= u-1,u+1}\nu_{\Lambda(u)}\left\vert \nabla_{r}
h_k
\right\vert^q+c_{1}\sum_{r= u-2,u+2}\nu_{\Lambda(u)}\left\vert \nabla_{r} h_k
\right\vert^q\end{align*}
 \end{proof}

~

 \begin{lemma} \label{lem5.4}Suppose conditions  (H1) is satisfied.
Then for  $r\in \Lambda(k)$, the following statements are true
 
~
 
\noindent  (a) When  $r=\{k-2,k,k+2\}$
  $$\nu\left\vert \nabla_{r} h_{k}
\right\vert^q\leq c_1\nu\left\vert \nabla_{r}f
\right\vert^q+\frac{J^qCc_{1}}{\epsilon }Q(k,k)+\frac{J^qc_{1}K}{\epsilon }\nu\left\vert f-\mathbb{E}^{\{\sim k\}}f
\right\vert^q$$
\noindent  (b) When  $r\in \{k-1,k+1\}$
$$\nu_{}\left\vert \nabla_{r} h_{k}
\right\vert^q= \nu\left\vert \nabla_{r}f
\right\vert^q$$where $h_k=f-\mathbb{E}^{\{\sim k\}}f$.\end{lemma}
 \begin{proof} We will show (a). For general $r \in \Lambda(k)\smallsetminus\{k-1,k+1\}$ we have
\begin{equation}\label{5.19}\nu\left\vert \nabla_r h_k
\right\vert^q\leq2^{q-1}\nu\left\vert \nabla_r
f
\right\vert^q+2^{q-1}\nu\left\vert \nabla_r
\mathbb{E}^{\{\sim k\}}f
\right\vert^q \end{equation}
We will now compute $\nu\left\vert \nabla_r
\mathbb{E}^{\{\sim k\}}f
\right\vert^q$ for the separate cases of $r\in \{k-2,k+2\}$  and $r=k$.
 
~
 
\noindent
 Consider  $r=\{k-2,k+2\}$. In this case
 \begin{align}\nonumber\nu\left\vert \nabla_r
\mathbb{E}^{\{\sim k\}}f
\right\vert^q\leq & 2^{q-1}\nu\left\vert \nabla_{r} f
\right\vert^q\\ &
\label{5.20}+J^q2^{q-1}\nu\left\vert\mathbb{E}^{\{\sim k\}}(f;\nabla_rV(x_s,x_r))\right\vert^q\mathcal{I}_{s\in \{k-1,k+1\}:s\sim r}\end{align}
 If we use Lemma~\ref{lem3.1} (a) to bound the second term on the right hand side of~\eqref{5.20} we obtain
 \begin{align} \nonumber\nu\left\vert \nabla_{r}
\mathbb{E}^{\{\sim k\}}f
\right\vert^q\leq&2^{q-1}\nu\left\vert \nabla_{r} f
\right\vert^q+\frac{J^q2^{q-1}C}{\epsilon }Q(k,k) \\& 
 \label{5.21}+\frac{J^q2^{q-1}K}{\epsilon }\nu\left\vert f-\mathbb{E}^{\{\sim k\}}f
\right\vert^q \end{align}
Combining~\eqref{5.19} and~\eqref{5.21} together we derive
$$\nu\left\vert \nabla_r h_k
\right\vert^q\leq c_1\nu\left\vert \nabla_rf
\right\vert^q+
 \frac{J^qCc_1}{\epsilon }Q(k,k)+\frac{J^qc_1K}{\epsilon}\nu\left\vert f-\mathbb{E}^{\{\sim k\}}f
\right\vert^q$$for $K$ as in (H2).

  Consider  $r=k$. In this case
 \begin{align}\nonumber\nu\left\vert \nabla_k
\mathbb{E}^{\{\sim k\}}f
\right\vert^q\leq & 2^{q-1}\nu_{}\left\vert \nabla_{k} f
\right\vert^q+J^q2^{q-1}\nu\left(\mathbb{E}^{\{\sim k\}}(f;W_{k})\right)^q \\   \leq&
\label{5.22} 2^{q-1}\nu\left\vert \nabla_{k}f
\right\vert^q+\frac{J^qC2^{q-1}}{\epsilon}Q(k,k)+\frac{J^q2^{q-1}K}{\epsilon}\nu\left\vert f-\mathbb{E}^{\{\sim k\}}f
\right\vert^q\end{align}
where in the last inequality the Lemma~\ref{lem3.1} (a) was used for $K$ as in (H2).  From~\eqref{5.19} and~\eqref{5.22}
$$\nu\left\vert \nabla_{k} h_{k}
\right\vert^q\leq c_1\nu\left\vert \nabla_{k}f
\right\vert^q+\frac{J^qCc_{1}}{\epsilon }Q(k,k)+\frac{J^qc_{1}K}{\epsilon }\nu\left\vert f-\mathbb{E}^{\{\sim k\}}f
\right\vert^q$$  \end{proof}
We can now prove Lemma~\ref{50}.

 ~

 \noindent
\textit{\textbf{Proof of Lemma~\ref{50}.}} If we combine the bound for $Q(k,k)$ from Lemma~\ref{lem5.3}, together with the bounds
for $\nu\left\vert \nabla_{r} h_{k}\right\vert^q,r=k-2,k-1,k,k+1,k+2$ from Lemma~\ref{lem5.4}, we obtain
\begin{align} Q(k,k)\leq& \sum_{r=k-1,k+1}\nu\left\vert \nabla_r f
\right\vert^q+c_1\nu\left\vert \nabla_kf
\right\vert^q+\frac{J^qCc_1}{\epsilon }Q(k,k)+\frac{J^qc_1K}{\epsilon}\nu\left\vert f-\mathbb{E}^{\{\sim k\}}f
\right\vert^q \nonumber  \\ & \nonumber
 +c_1\sum_{r=k-2,k+2}\left( c_1\nu\left\vert \nabla_rf
\right\vert^q+\frac{J^qCc_1}{\epsilon }Q(k,k)+\frac{J^qc_1K}{\epsilon }\nu\left\vert f-\mathbb{E}^{\{\sim k\}}f\right\vert^q \right)\\  & \nonumber
 +\frac{J^qc_1C}{\epsilon}\sum_{dist(k,t)=4}Q(t,k)+\frac{J^qc_12K}{\epsilon}\nu\left\vert f-\mathbb{E}^{\{\sim k\}}f\right\vert^q \\= & \nonumber \sum_{r=k-1,k+1}\nu\left\vert \nabla_r f
\right\vert^q+c_1\nu\left\vert \nabla_kf
\right\vert^q+\frac{J^q(c_13+c^2_12)K}{\epsilon}\nu\left\vert f-\mathbb{E}^{\{\sim k\}}f\right\vert^q\\ &  +  \label{5.23}c^2_1\sum_{r=k-2,k+2} \nu\left\vert \nabla_rf
\right\vert^q+\frac{J^q2C(c_1+c^2_1)}{\epsilon}Q(k,k)+\frac{J^qc_1C}{\epsilon}\sum_{dist(k,t)=4}Q(t,k)\end{align}
In order to bound $\sum_{dist(k,t)=4}Q(t,k)$ in the above quantity the lemma
bellow will be used.
 
\begin{lemma} \label{lem5.5}Under conditions (H1)-(H3) the following inequality
 \begin{align*}\sum_{t:dist(t,k)=4}Q(t,k)\leq & J^qTQ(k,k)+J^qT\nu\left\vert f-\mathbb{E}^{\{\sim k\}}f\right\vert^q+T\sum_{r=k-2,k+2}\nu\left\vert \nabla_r
f
\right\vert^q\\ &
 +T\sum_{n=0}^{\infty} J^{n(q-1)}\sum_{r=0}^3\left(\nu\left\vert \nabla_{k+3+4n+r} f
\right\vert^q+\nu\left\vert \nabla_{k-3-4n-r} f
\right\vert^q\right)\end{align*}
 is satisfied for some positive constant $T$ independent of $k$.\end{lemma}

The proof of Lemma~\ref{lem5.5} will be presented later in the section. If we use the bound of Lemma~\ref{lem5.5} in~\eqref{5.23}, we
obtain
 \begin{align}Q(k,k)\leq&\nonumber J^q\left(\frac{TJ^qc_{1}C}{ \epsilon}+\frac{c_13cK}{\epsilon}+\frac{c^2_12cK}{\epsilon}\right)\nu\left\vert f-\mathbb{E}^{\{\sim k\}}f
\right\vert^q\\ &
 \nonumber +J^q\left(\frac{2Cc^2_1}{\epsilon}+\frac{Cc_1}{\epsilon}\right )Q(k,k)+\frac{J^qc_1C}{ \epsilon}J^qTQ(k,k)\\ &\nonumber
+\sum_{r=k-1,k+1}\nu\left\vert \nabla_r f
\right\vert^q+c_1\nu\left\vert \nabla_kf
\right\vert^q+( \frac{J^qc_1C}{ \epsilon }T+c_1^2)\sum_{r=k-2,k+2}\nu\left\vert \nabla_r f
\right\vert^q \\ &
 \label{5.24}+\frac{J^qc_1C}{ \epsilon}T\sum_{ n=0 }^{\infty} J^{n(q-1)}\sum_{r=0}^3\left(\nu\left\vert \nabla_{k+3+4n+r} f
\right\vert^q+\nu\left\vert \nabla_{k-3-4n-r} f
\right\vert^q\right)\end{align}
 If we choose $J$ sufficiently small such that
 $$1-J^q\left(\frac{2Cc^2_1}{\epsilon}+\frac{J^qc_1CT}{ \epsilon}+\frac{Cc_1}{\epsilon }\right )>\frac{1}{2}$$
then from~\eqref{5.24} we have
\begin{align*}Q(k,k)\leq&2J^q\left(\frac{TJ^qc_{1}C}{ \epsilon}+\frac{c_13cK}{\epsilon}+\frac{c^2_12cK}{\epsilon}\right)\nu\left\vert f-\mathbb{E}^{\{\sim k\}}f
\right\vert^q+2c_1\nu\left\vert \nabla_kf
\right\vert^q\\ &
\nonumber
+2\sum_{r=k-1,k+1}\nu\left\vert \nabla_r f
\right\vert^q+2( \frac{J^qc_1C}{ \epsilon }T+c_1^2)\sum_{r=k-2,k+2}\nu\left\vert \nabla_r f
\right\vert^q \\ &+
\frac{2J^qc_1C}{ \epsilon}T\sum_{ n=0 }^{\infty} J^{n(q-1)}\sum_{r=0}^3\left(\nu\left\vert \nabla_{k+3+4n+r} f
\right\vert^q+\nu\left\vert \nabla_{k-3-4n-r} f
\right\vert^q\right)\end{align*}
 and the lemma follows for an appropriate positive constant $D$. \qed

 ~

  It remains to show Lemma~\ref{lem5.5}. For this we will need the following
 lemmata. 
 \begin{lemma} \label{lem5.6}Under conditions (H1)-(H3) the following two
bounds for $Q(u,k)$ hold.
 
~

\noindent (a) For $u$ such that $dist(u,k)\geq 8$
   \begin{align*}Q(u,k)\leq & c_1\nu_{\Lambda(u)}\left\vert \nabla_u
f
\right\vert^q+c_1\sum_{r= u-1,u+1}\nu_{\Lambda(u)}\left\vert \nabla_r
f
\right\vert^q+c^2_1\sum_{r=u-2,u+2}\nu_{\Lambda(u)}\left\vert \nabla_r f
\right\vert^q \\ &
+ 
\frac{J^qc_1C}{\epsilon}\sum_{dist(u,t)=4}Q(t,k)+\frac{J^qc_{1}2K}{\epsilon}\nu\left\vert f-\mathbb{E}^{\{\sim k\}}f
\right\vert^q\end{align*}
(b) For $u$ such that $dist(u,k)=4$
\begin{align}Q(u,k)\nonumber\leq &c_1\nu\left\vert \nabla_u
f\right\vert^q+c^2_1\sum_{r=u-2,u+2}\nu\left\vert \nabla_r f
\right\vert^q
\nonumber +J^q\left(\frac{c_12K}{\epsilon}+\frac{c^2_1K}{\epsilon }\right)\nu\left\vert f-\mathbb{E}^{\{\sim k\}}f
\right\vert^q \\ &
 \nonumber +c_1\sum_{r=u-1,u+1}\nu\left\vert \nabla_r
f
\right\vert^q+\frac{J^qCc^2_1}{\epsilon}Q(k,k)+\frac{J^qc_1C}{\epsilon}\sum_{dist(u,t)=4,t\neq k}Q(t,k)\end{align}\end{lemma} \begin{proof} The lemma follows from the bound of $Q(u,k)$ in Lemma~\ref{lem5.3}.
In the case where $dist(u,k)\geq 8$,  for $r=u-2,u-1,u,u+1,u+2$ we have
that
\begin{equation}\label{5.25}\nu\left\vert \nabla_{r}
h_k
\right\vert^q\leq2^{q-1}2\nu\left\vert \nabla_{r}
f
\right\vert^q\end{equation}
Substituting~\eqref{5.25} in the expression from Lemma~\ref{lem5.3} we immediately obtain (a).
Consider the case where  $dist(u,k)=4$. Then for $r=u-1,u,u+1$
\begin{equation}\label{5.26}\nu\left\vert \nabla_{r}
h_k
\right\vert^q\leq2^{q-1}2\nu\left\vert \nabla_{r}
f
\right\vert^q\end{equation}
While for $r=\{u-2,u+2\}$ we can bound  $\nu\left\vert \nabla_{r} h_{k}
\right\vert^q$ from  Lemma~\ref{lem5.4} (a). If we plug the bounds from~\eqref{5.26} and~Lemma~\ref{lem5.4} (a) into  the expression from Lemma~\ref{lem5.3},  we obtain
\begin{align*}Q(u,k)\leq &
J^q\left(\frac{c_12K}{\epsilon}+\frac{c^2_1K}{\epsilon }\right)\nu\left\vert f-\mathbb{E}^{\{\sim k\}}f
\right\vert^q+
 \frac{J^qCc^2_1}{\epsilon}Q(k,k)+c_1\nu\left\vert \nabla_u
f
\right\vert^q\\ &
+c_1\sum_{r= u-1,u+1}\nu\left\vert \nabla_r
f
\right\vert^q+c^2_1\sum_{r=u-2,u+2}\nu\left\vert \nabla_r f
\right\vert^q+ \frac{J^qc_1C}{ \epsilon}\sum_{dist(u,t)=4}Q(t,k)\end{align*} \end{proof} 
Before proving Lemma \ref{lem5.5}, we will also need to show that for any $k\in \mathbb{N}$$$\sup_{n\in\mathbb{N}}\sum_{dist(u,k)=n} Q(u,k)<C_{f}<\infty$$for $C_f$  a constant which depends on the function $f$ but not on $n,u$ and $k$. To show this we first need the following lemma. 
\begin{lemma}\label{newlemma0section5}For any $r,k \in \mathbb{Z}$ we have $$\nu\left\vert \nabla_r h_k
\right\vert^q\leq\tilde C_f<\infty$$
where $\tilde C_f$ depends on the function $f$ but not on $r$ and $k$.
\end{lemma}
\begin{proof}For general $r \in \{k-2,k,k+2\}$ \begin{equation}\label{neweq1ofnewlemma10}\nu\left\vert \nabla_r h_k
\right\vert^q\leq2^{q-1}\nu\left\vert \nabla_r
f
\right\vert^q+2^{q-1}\nu\left\vert \nabla_r
\mathbb{E}^{\{\sim k\}}f
\right\vert^q \end{equation}
since $h_k=f-\mathbb{E}^{\{\sim k\}}f$. For the second term on the right hand side of (\ref{neweq1ofnewlemma10}) we have
 \begin{align}\label{neweq2ofnewlemma10}\nu\left\vert \nabla_r
\mathbb{E}^{\{\sim k\}}f
\right\vert^q\leq & 2^{q-1}\nu\left\vert \nabla_{r} f
\right\vert^q+J^q2^{q-1}\nu\left\vert\mathbb{E}^{\{\sim k\}}(f;Z_k)\right\vert^q \end{align}
where $$Z_k=\nabla_{k-2}V(x_{k-2},x_{k-1})\mathcal{I}_{r=k-1}+\nabla_{k+2}V(x_{k+2},x_{k+1})\mathcal{I}_{r=k+1}+W_{k}\mathcal{I}_{r=k}$$
where $W_k$ as in (\ref{defineW_k}). We will now compute the last term on the right hand side of (\ref{neweq2ofnewlemma10})
\begin{align*} \nu\left\vert\mathbb{E}^{\{\sim k\}}(f;Z_k)\right\vert^q =&\nu\left\vert\mathbb{E}^{\{\sim k\}}(f-\mathbb{E}^{\{\sim k\}}f)(Z_{k}-\mathbb{E}^{\{\sim k\}}Z_k)\right\vert^q \\ =&\nu\left\vert\mathbb{E}^{\{\sim k\}}\left(f(Z_{k}-\mathbb{E}^{\{\sim k\}}Z_k)\right)\right\vert^q  \leq \nu f^{q}\vert Z_{k}-\mathbb{E}^{\{\sim k\}}Z_k\vert^q\end{align*}
If we use the entropic inequality (\ref{3.3}) we obtain
\begin{align} \label{neweq3ofnewlemma10}\nu\left\vert\mathbb{E}^{\{\sim k\}}(f;Z_k)\right\vert^q \leq\frac{1}{\epsilon} &\nu f^q\log\frac{f^q}{\nu f^q}+\frac{1}{\epsilon}\nu f^q\log \nu e^{\epsilon\vert Z_{k}-\mathbb{E}^{\{\sim k\}}Z_k\vert^q}\nonumber\\ \leq\frac{1}{\epsilon} &\nu f^q\log\frac{f^q}{\nu f^q}+\frac{K}{\epsilon}\nu f^q\end{align}
   where $K$ as in (H2). If we combine (\ref{neweq1ofnewlemma10}), (\ref{neweq2ofnewlemma10}) and (\ref{neweq3ofnewlemma10})  we get
that for $r \in \{k-2,k,k+2\}$
\begin{align}\label{neweq3-4ofnewlemma10} \nu\left\vert \nabla_r h_k
\right\vert^q\leq2^{q}\nu\left\vert \nabla_r
f
\right\vert^q+\frac{J^q 2^{2q-2}}{\epsilon} \nu f^q\log\frac{f^q}{\nu f^q}+\frac{J^q 2^{2q-2}}{\epsilon}\nu f^q
\end{align}
For $r \notin \{k-2,k,k+2\}$ we have 
\begin{equation}\label{neweq3-4+1ofnewlemma10}\nu\left\vert \nabla_r h_k
\right\vert^q\leq2^{q} \nu f^q
\end{equation}
From (\ref{neweq3-4ofnewlemma10}) and (\ref{neweq3-4+1ofnewlemma10}) the lemma follows since  functions $f$ are as in Remark \ref{newremark1}. 
 \end{proof} 
\begin{lemma}\label{newlemmasection5}If (H2) is satisfied, then for any $k\in \mathbb{N}$$$\sup_{n\in\mathbb{N}}\sum_{dist(u,k)=n} Q(u,k)<C_{f}<\infty$$where $C_f$ is a constant which depends on the function $f$ but not on $u$ and $k$. \end{lemma}
\begin{proof}
Since we work on the one dimensional lattice, it  is sufficient to show that
$$\sup_{n\in\mathbb{N}} Q(u,k)<C'_{f}<\infty$$
 for $C'_f$ depends only on the functions $f$. To  compute $Q(u,k)$ we can use  (\ref{5.14}) and (\ref{5.15}) to obtain\begin{align}\label{neweq4ofnewlemma10}Q(u,k)\leq\sum_{r= u-1,u,u+1}\nu\left\vert \nabla_r
h_k
\right\vert^q
+\sum_{r= u-2,u+2}\nu_{\Lambda(u)}\left\vert \nabla_{r} (\mathbb{E}^{M(u)}
\vert h_k\vert^q)^{\frac{1}{q}}
\right\vert^q \end{align}
 Furthermore, from (\ref{5.16}) for $r=u-2,u+2$ we have  \begin{align}\label{neweq6ofnewlemma10}\nu_{\Lambda(u)}&\left\vert \nabla_{r} (\mathbb{E}^{M(u)}
\vert h_k\vert^q)^{\frac{1}{q}}
\right\vert^q\leq c_1\nu\left\vert \nabla_{r} h_{k}
\right\vert^q+
 \frac{J^qc_1}{q^q}I_0
\end{align} 
where $$I_0 :=\nu_{\Lambda(u)}\left(\mathbb{E}^{M(u)}\vert h_k\vert^{q}\right)^{-\frac{q}{p}}\left(\mathbb{E}^{M(u)}(\vert h_k\vert^q;\nabla_rV(x_{r},x_{s}))\right)^q
 \mathcal{I}_{s\in \{u-3,u+3\}:s\sim r}$$ 
 In order to bound  the second term on the right hand side of (\ref{neweq6ofnewlemma10})
we compute 
\begin{align*}\mathbb{E}^{M(u)}&(\vert h_k\vert^q;\nabla_rV(x_{r},x_{s}))=\mathbb{E}^{M(u)}\left(\vert h_k\vert^{(q-1)+1}\left(\nabla_rV(x_{r},x_{s})-\mathbb{E}^{M(u)}\nabla_rV(x_{r},x_{s})\right)\right)\\&\leq\left(\mathbb{E}^{M(u)}\vert h_k\vert^{pq-p}\right)^{\frac{1}{p}} \left( \mathbb{E}^{M(u)}\left(\vert h_k\vert^{q}\left\vert\nabla_rV(x_{r},x_{s})-\mathbb{E}^{M(u)}\nabla_rV(x_{r},x_{s})\right\vert^q\right) \right)^\frac{1}{q}
\end{align*}
From the last bound, since $p$ and $q$ are conjugate, we get \begin{align*}I_{0}& \leq\nu_{\Lambda(u)}\mathbb{E}^{M(u)}\left(\vert h_k\vert^{q}\left\vert\nabla_rV(x_{r},x_{s})-\mathbb{E}^{M(u)}\nabla_rV(x_{r},x_{s})\right\vert^q\right)
 \mathcal{I}_{s\in \{u-3,u+3\}:s\sim r}\\&=\nu(\vert h_k\vert^{q}N_r)\leq 2^{q-1}\nu (f^{q}N_r)+2^{q-1}\nu( (\mathbb{E}^{\{\sim k\}} f^{q})N_r)\end{align*}
 where above we denoted $N_{r}=\left\vert\nabla_rV(x_{r},x_{s})-\mathbb{E}^{M(u)}\nabla_r V(x_{r},x_{s})\right\vert^q
 \mathcal{I}_{s\in \{u-3,u+3\}:s\sim r}$. If we use again the entropic inequality (\ref{3.3}) we obtain
 \begin{align}\nonumber\label{neweq7ofnewlemma10}I_{0}\nonumber \leq&\frac{2^{q-1}   }{\epsilon}\nu f^{q}\log \frac{f^{q}}{\nu f^{q}}+\frac{2^{q-1}   }{\epsilon}\nu f^{q}\log \nu e^{\epsilon N_r}+\frac{2^{q-1}   }{\epsilon}\nu\mathbb{E}^{\{\sim k\}} f^{q}\log\frac{ \mathbb{E}^{\{\sim k\}} f^{q}}{\nu \mathbb{E}^{\{\sim k\}} f^{q}}\\ \nonumber &+\frac{2^{q-1}   }{\epsilon}\log \nu e^{\epsilon N_r}\nu \mathbb{E}^{\{\sim k\}} f^{q}\\ \leq& \frac{2^{q-1}   }{\epsilon}\nu f^{q}\log \frac{f^{q}}{\nu f^{q}}+\frac{2^{q} K  }{\epsilon}\nu f^{q}+\frac{2^{q-1}   }{\epsilon}\nu\mathbb{E}^{\{\sim k\}} f^{q}\log\frac{ \mathbb{E}^{\{\sim k\}} f^{q}}{\nu \mathbb{E}^{\{\sim k\}} f^{q}}\end{align}
where $K$ as in (H2).
For the last term on the right hand side of (\ref{neweq7ofnewlemma10}) we can  write
\begin{align}\label{neweq7+1ofnewlemma10}\nu\mathbb{E}^{\{\sim k\}} f^{q}\log\frac{ \mathbb{E}^{\{\sim k\}} f^{q}}{\nu \mathbb{E}^{\{\sim k\}} f^{q}}=\nu f^{q}\log\frac{ \mathbb{E}^{\{\sim k\}} f^{q}}{\nu \mathbb{E}^{\{\sim k\}} f^{q}}\leq\nu f^{q}\log\frac{ f^{q}}{\nu f^{q}}
\end{align}
Combining together (\ref{neweq7ofnewlemma10}) and (\ref{neweq7+1ofnewlemma10})
 we obtain \begin{align}\label{neweq8ofnewlemma10} I_{0}\leq\frac{2^{q}   }{\epsilon}\nu f^{q}\log \frac{f^{q}}{\nu f^{q}}+\frac{2^{q} K  }{\epsilon}\nu f^{q}
\end{align}
From (\ref{neweq6ofnewlemma10}), and     (\ref{neweq8ofnewlemma10}) we then get that for $r=u-2,u+2$   \begin{align}\label{neweq9ofnewlemma10}\nu_{\Lambda(u)}\left\vert \nabla_{r} (\mathbb{E}^{M(u)}
\vert h_k\vert^q)^{\frac{1}{q}}
\right\vert^q\leq c_1\nu\left\vert \nabla_{r} h_{k}
\right\vert^q+
 \frac{J^q2^{q}c_1}{q^q\epsilon}\nu f^{q}\log \frac{f^{q}}{\nu f^{q}}+\frac{J^qc_12^{q} K  }{q^{q}\epsilon}\nu f^{q}
\end{align} 
If we combine  (\ref{neweq4ofnewlemma10}) and (\ref{neweq9ofnewlemma10}) together with Lemma \ref{newlemma0section5}
 we conclude that   for any function $f$ there is a bound of $\nu_{\Lambda(u)}\left\vert \nabla_{r} (\mathbb{E}^{M(u)}
\vert h_k\vert^q)^{\frac{1}{q}}
\right\vert^q$ uniformly with respect to the set $M(u)$ depending only on $\nu f^{q}$, $\max_{i\in \mathbb{Z}}\nu \left \vert \nabla _i f \right\vert^q$ and  $\nu f^{q}\log \frac{f^{q}}{\nu f^{q}}$.

\end{proof}
 We can now prove Lemma~\ref{lem5.5}.

 ~

 \noindent
\textit{\textbf{Proof of Lemma~\ref{lem5.5}.}} For every $u$ s.t. $dist(u,k)\geq8$ define  
  \begin{align*}G(u,k):=& c_1\nu_{\Lambda(u)}\left\vert \nabla_u
f
\right\vert^q+c_1\sum_{r= u-1,u+1}\nu_{}\left\vert \nabla_r
f
\right\vert^q \\ &
+c^2_1\sum_{r=u-2,u+2}\nu_{\Lambda(u)}\left\vert \nabla_r f
\right\vert^q+\frac{J^qc_12K}{\epsilon}\nu\left\vert f-\mathbb{E}^{\{\sim k\}}f
\right\vert^q\end{align*}
 and for every $u$ s.t. $dist(u,k)=4$ define
\begin{align*} G(u,k):=&c_1\nu\left\vert \nabla_{u}
f
\right\vert^q+c_1\sum_{r= u-1,u+1}\nu\left\vert \nabla_u
f
\right\vert^q+\frac{J^qCc^2_1}{\epsilon}Q(k,k)\\ &
 +c^2_1\sum_{i=u-2,u+2}\nu\left\vert \nabla_r f
\right\vert^q+J^q\left(\frac{c_12K}{\epsilon}+\frac{c^2_1K}{\epsilon }\right)\nu\left\vert f-\mathbb{E}^{\{\sim k\}}f
\right\vert^q\end{align*}
If we set
$K'=\frac{c_{1}C}{ \epsilon}$, then from Lemma~\ref{lem5.6} (a) and (b) respectively
we can write
\begin{equation} \label{5.28}Q(u,k)\leq G(u,k)+ J^qK'\sum_{dist(u,t)=4}Q(t,k)  \text{,
\; for } dist(u,k)\geq8\end{equation}
and
\begin{equation}
\label{5.29}Q(u,k)\leq G(u,k)+ J^qK'Q(t,k)\mathcal{I}_{dist(t,u)=4,t\neq k} \text{, \;
for } dist(u,k)=4 \end{equation}
 From equation~\eqref{5.28} we obtain
  $$\sum_{dist(u,k)=n} Q(u,k)\leq \sum_{dist(u,k)=n}G(u,k)+ J^qK'\sum_{dist(u,k)=n}\sum_{dist(t, u)=4}Q(t,k)$$
  or equivalently
 \begin{align*}\sum_{dist(u,k)=n} Q(u,k)\leq &\sum_{dist(u,k)=n}G(u,k)+ J^qK'\sum_{dist(t,k)=n+4}Q(t,k)\\
 &  +J^qK'\sum_{dist(t,k)=n-4}Q(t,k)\end{align*}
 which implies
\begin{equation}\label{5.30}\tilde{Q}(n)\leq \tilde{G}(n)+J^q K'\tilde{Q}(n-4)+J^qK'\tilde{Q}(n+4)\end{equation}
where we denote $$\tilde{Q}(n)=\sum_{dist(u,k)=n} Q(u,k)\text{\; and \;}\tilde{G}(n)= \sum_{dist(u,k)=n}G(u,k)$$
While from equation~\eqref{5.29}, we have
 $$\sum_{dist(u,k)=4} Q(u,k)\leq \sum_{dist(u,k)=4}G(u,k)+ J^qK'\sum_{dist(u,k)=4}Q(t,k)\mathcal{I}_{dist(t,u)=4,t\neq k}$$
 This implies
 $$\sum_{dist(u,k)=4} Q(u,k)\leq \sum_{dist(u,k)=4}G(u,k)+ J^qK'\sum_{dist(t,k)=8}Q(t,k)$$
 which is equivalent to
 \begin{equation}\label{5.31}\tilde{Q}(4)\leq \tilde{G}(4)+J^qK'\tilde{Q}(8)  \end{equation}
 Choose  $J$ in (H3)
sufficiently small such that hypothesis~\eqref{l6-5.8} of Lemma~\ref{lem5.2} is satisfied. Then, since  relationships~\eqref{5.30}, ~\eqref{5.31} and Lemma \ref{newlemmasection5} are true, the conditions of Lemma~\ref{lem5.2}  are satisfied  for $P=\tilde{ Q}$ and  $G=\tilde{ G}$ and so we  obtain
$$\nonumber \tilde{Q}(4)\leq \hat{ J}\sum_{n=0}^{+\infty} J^{nq-n}\tilde{G}(4n+4)$$
where $\hat{ J}=\frac{1}{1-J^{2q-2}}$. This is equivalent to
\begin{align}  \sum_{t:dist(t,k)=4}Q(t,k)\leq &\hat{ J} \sum_{dist(u,k)=4}G(u,k)\nonumber\\
  &   \label{5.32}+  \hat{ J}\sum_{n=1}^{+\infty} J^{nq-n}\sum_{dist(u,k)=4n+4}G(u,k)\end{align}
  Substituting $G(u,k)$     leads to
 \begin{align}\nonumber\sum_{t:dist(t,k)=4}Q&(t,k)\leq\frac{J^qCc^{2}_{1}}{\epsilon }\hat J \sum_{dist(u,k)=4}Q(k,k)\\  &
 \nonumber + \hat{ J}c_1\sum_{n=0}^{+\infty} J^{nq-n}\sum_{dist(u,k)=4n+4} \nu_{\Lambda(u)}\left\vert \nabla_{u}
f
\right\vert^q\\  &
 + \hat{ J}c_1\sum_{n=0}^{+\infty} J^{nq-n}\sum_{dist(u,k)=4n+4}\sum_{r= u-1,u+1}\nu_{\Lambda(u)}\left\vert \nabla_{r}
f
\right\vert^q\nonumber\\ \nonumber  &
 +\hat{ J}c^2_1\sum_{n=0}^{+\infty} J^{nq-n}\sum_{dist(u,k)=4n+4}\sum_{r= u-2,u+2}\nu_{\Lambda(u)}\left\vert \nabla_{r} f
\right\vert^q\\  &
 \label{5.33}+J^q\hat J\frac{c_1K}{\epsilon}(c_{1}+2)\sum_{n=0}^{+\infty} J^{nq-n}\sum_{dist(u,k)=4n+4}\nu\left\vert f-\mathbb{E}^{\{\sim k\}}f
\right\vert^q\end{align}
But for   $J$   in  (H3)  we have    $J^{q-1}<1$ which implies $\tilde{J}=\sum_{n=0}^{+\infty} J^{nq-n}<\infty$. ~\eqref{5.33} then implies

~

$\sum_{t:dist(t,k)=4}Q(t,k)\leq\frac{J^qCc^{2}_{1}}{\epsilon}2\hat{ J} Q(k,k)$

~

$ \ \ \ \ \  \ \ \ \ +\hat{J}c_1\tilde{J}\sum_{dist(u,k)=4n+4}\sum_{r= u-1,u,u+1}\nu\left\vert \nabla_{r}
f
\right\vert^q$

~

$ \ \ \ \ \ \ \ \ \  +\hat{J}c^2_1\tilde{J}\sum_{dist(u,k)=4n+4}\sum_{r= u-2,u+2}\nu\left\vert \nabla_{r} f
\right\vert^q$

~

$ \ \ \ \ \ \ \ \ +J^{q}\frac{c_1K}{\epsilon}(c_{1}+2)\hat{ J}\frac{2}{1-J^{q-1}}\nu\left\vert f-\mathbb{E}^{\{\sim k\}}f
\right\vert^q $
 
~
 
\noindent and the lemma follows for appropriate constant $T>0$. \qed
 \section{Conclusion}
 In the present work,  we have determined conditions
 for the infinite volume Gibbs measure  to satisfy the Log-Sobolev
 Inequality. As explained in the introduction, the   criterion presented
 in Theorem~\ref{thm2.1} can     in particular be applied in the case of local specifications $\{\mathbb{E}^{\Lambda,\omega}\}_{\Lambda\subset\subset \mathbb{Z} ,\omega \in
\Omega}$    with   no quadratic interactions for which $$\left\Vert \nabla_i \nabla_j V(x_i,x_j) \right\Vert_{\infty}=\infty$$
Thus, we have shown that our results  can
go    beyond the usual uniform boundness of the second derivative of the interactions considered in  [Z1], [Z2],  [M] and [O-R].

Concerning the additional conditions (H1) and (H2) placed
here to handle the exotic interactions, they refer to finite dimensional measures with no boundary conditions which
are easier to handle than the  $\{\mathbb{E}^{\Lambda,\omega}\}_{\Lambda\subset\subset \mathbb{Z} ,\omega \in
\Omega}$  measures or the
infinite dimensional Gibbs measure $\nu$. 

In fact, the following results concerning the conditions can be proven. This is a work in progress that will consist the material of a forthcoming paper. \begin{proposition}\label{conclProp}The hypothesis (H0), (H3) and (H2) imply hypothesis (H1). \end{proposition} 
Consequently,    the main result of Theorem \ref{thm2.1} is then reduced to the following  \begin{theorem} \label{conclTheor}If hypothesis (H0), (H3) and (H2) are satisfied, then the infinite dimensional Gibbs measure $\nu$  for the local specification $\{\mathbb{E}^{\Lambda,\omega}\}_{\Lambda\subset\subset \mathbb{Z},\omega \in
\Omega}$ satisfies the $q$ Log-Sobolev inequality
$$\nu \left\vert f\right\vert^q log\frac{\left\vert f\right\vert^q}{\nu \left\vert f\right\vert^q}\leq \mathfrak{C} \ \nu \left\vert \nabla f
\right\vert^q$$                              
for some positive constant $\mathfrak{C}$ independent of $f$. \end{theorem}
 Concerning examples of measures that satisfy
 the above conditions, one can consider measures with phase  $\phi (x)=\vert x\vert^t$ with $t\geq \frac{q}{q-1}$ and interaction
 $V(x,y)=\vert x-y\vert^r$, with    $\max\{r,(r-1)q\}<t$.   
The main idea of the proof of the  Proposition  \ref{conclProp} follows in main lines the method followed in the current paper. Although some of the details are more involved because of the lack of hypothesis (H1), the fact that in Proposition \ref{conclProp}  the Gibbs measure is localised and thus the approximation procedure starts from a finite set  compensates for the loss of the LSq for  $\nu_{\Lambda (i)}$.

In this paper we have been concerned with the $q$ Logarithmic Sobolev inequality  for measures on the 1 dimensional Lattice $\mathbb{Z}$. It is interesting to try to extend the current result to a higher dimensional lattice on $\mathbb{Z}^d,d\geq 2$, although this does not appear to be immediate. In a different direction,    we can consider the  following  class of modified  Logarithmic Sobolev inequalities  presented in [G-G-M]: 
\begin{equation}\label{modeq}\nu \left\vert f\right\vert^2 log\frac{\left\vert f\right\vert^2}{\nu \left\vert f\right\vert^2}\leq  \mathfrak{C}\  \mathfrak{}\ \int  H_{a,c}  \left( \frac{ \nabla f
}{f}\right)f^2d\nu\end{equation}                            
for  some positive constant $\mathfrak{C}$, where  $$ H_{a,c}( x)=\begin{cases}\frac{x^2}{2} & \text{if \ } \vert  x\vert \leq a \\
a^{2-\beta} \frac{\vert x \vert ^{\beta}}{\beta}+a^2\frac{\beta-2}{2\beta}& \text{if \ }\vert x\vert \geq\ a \text{\ and \ }c \neq1\\
+\infty & \text{if \ }\vert x\vert \geq\ a \text{\ and \ }c =1
\end{cases}$$
for $c\in [1,2], a>0$ and $\beta$ satisfying $\frac{1}{c}+\frac{1}{\beta}=1$ ($\beta\geq 2$).
This new class of inequalities is an interpolation between  Log-Sobolev (LS2) and Spectral Gap inequalities (SG2), which retains the basic  properties of the Log-Sobolev inequalities mentioned in Remark \ref{rem1.1}. Some preliminary results suggest that  on $\mathbb{Z}^d,d\geq2$, the infinite dimensional Gibbs measure  satisfies a [G-G-M] type   inequality with $\beta=2q$,   under hypothesis (H0) for LSq ($1< q<  2$) and some  hypothesis stronger than (H2). This is work in early stages, but   hopefully a  modified LS inequality comparable to the [G-G-M] inequalities can be obtained  in the case of  the higher dimensional lattice.

In addition, it is interesting to investigate whether the result presented in this paper can be extended to the family of weaker inequalities presented in [G-G-M], assuming (H0) and (H1) for the (\ref{modeq}) inequality instead of the LSq.  However, this does not seem to be immediate especially in showing the sweeping out relationships and so more work needs to be done towards this direction.

Furthermore, concerning the  hypothesis on the single-site measure,    the main hypothesis (H0) for $\mathbb{E}^{\{i\},\omega}$ can be reduced to the same assumption for the boundary free single-site measure, that
is

~

\noindent \textbf{(H0$'$)}: The single-site measure $\frac{e^{-\phi (x)}dx}{\int e^{-\phi (x)}dx}$ satisfies
the LSq Inequality.

~

 Measures as in ($H0'$) do not involve boundary conditions and for this
reason it is easier to show that they satisfy the Log-Sobolev inequality.
 For instance,  when in $\mathbb{R}$ one can think of
 phases    that are  convex  and increase sufficiently fast, like     $\phi(x)=\left\vert x\right\vert^p$ for $p>2$ (see [B-Z]). In the case of the Heisenberg
   group $\mathbb{}\mathbb{H}$    one can consider $\phi(x)=\beta d(x)^p$
   with $p$ conjugate of $q$ (see [H-Z]).  

However, that does not mean that condition (H0$'$) is in general weaker  than condition
(H0) as there are examples of single-site boundary free
measures $\frac{e^{-\phi (x)}dx}{\int e^{-\phi (x)}dx}$ that do not satisfy
the LS$q$ inequality, which    when perturbed with   interactions, give  new measures
 $\mathbb{E}^{\{i\},\omega}$ that satisfy the Log-Sobolev-q  inequality uniformly on the boundary conditions, that is condition (H0) is satisfied.  In addition, in the case of hypothesis ($H0'$), it seems that the analogues of Proposition \ref{conclProp} and Theorem \ref{conclTheor} will be more to difficult to be shown.

 ~

 \noindent\textbf{Acknowledgements:} The author would like to thank Prof. Boguslaw Zegarlinski for his valuable comments and suggestions.

\bibliographystyle{alpha}

\end{document}